\documentclass[a4paper]{amsart}

\usepackage{esint}
\usepackage[abbrev,backrefs]{amsrefs}
\usepackage{amssymb}
\usepackage{algpseudocode}
\usepackage{algorithm}

\makeatletter
\newenvironment{breakablealgorithm}
  {
   \begin{center}
     \refstepcounter{algorithm}
     \hrule height.8pt depth0pt \kern2pt
     \renewcommand{\caption}[2][\relax]{
       {\raggedright\textbf{\fname@algorithm~\thealgorithm} ##2\par}%
       \ifx\relax##1\relax 
         \addcontentsline{loa}{algorithm}{\protect\numberline{\thealgorithm}##2}%
       \else 
         \addcontentsline{loa}{algorithm}{\protect\numberline{\thealgorithm}##1}%
       \fi
       \kern2pt\hrule\kern2pt
     }
  }{
     \kern2pt\hrule\relax
   \end{center}
  }
\makeatother

\newtheorem{theorem}{Theorem}[section]
\newtheorem{proposition}{Proposition}[section]
\newtheorem{lemma}[theorem]{Lemma}

\theoremstyle{definition}
\newtheorem{definition}[theorem]{Definition}

\def\restrict#1{\raise-.5ex\hbox{\ensuremath|}_{#1}}

\theoremstyle{remark}
\newtheorem{remark}[theorem]{Remark}

\numberwithin{equation}{section}

\begin{document}

\title{Fractional Integration and Optimal Estimates for Elliptic Systems}


\author[F. Hernandez]{Felipe Hernandez}
\address[F. Hernandez]{Department of Mathematics,
Building 380, Stanford, California 94305}
\email{felipehb@stanford.edu}

\author[D. Spector]{Daniel Spector}
\address[D. Spector]{Department of Mathematics, National Taiwan Normal University, No. 88, Section 4, Tingzhou Road, Wenshan District, Taipei City, Taiwan 116, R.O.C.;\newline
Okinawa Institute of Science and Technology Graduate University,
Nonlinear Analysis Unit, 1919--1 Tancha, Onna-son, Kunigami-gun,
Okinawa, Japan
}
\email{spectda@protonmail.com}



\subjclass[2010]{Primary }

\date{}

\dedicatory{}

\commby{}

\begin{abstract}
In this paper we give an affirmative answer to the Euclidean analogue of a question of Bourgain and Brezis concerning the optimal Lorentz estimate for a Div-Curl system:  The function $Z=\operatorname*{curl} (-\Delta)^{-1} F$ satisfies
\begin{align*}
\operatorname*{curl} Z &= F   \\
 \operatorname*{div} Z &= 0
\end{align*}
and there exists a constant $C>0$ such that
\begin{align*}
\| Z\|_{L^{3/2,1}(\mathbb{R}^3;\mathbb{R}^3)} \leq  C\| F\|_{L^{1}(\mathbb{R}^3;\mathbb{R}^3)}.
\end{align*}
Our proof relies on a new endpoint Hardy-Littlewood-Sobolev inequality for divergence free measures which we obtain via a result of independent interest, an atomic decomposition of such objects.
\end{abstract}

\maketitle

\section{Introduction}
Let $F \in L^1(\mathbb{R}^3;\mathbb{R}^3)$ be a given divergence free vector field and consider the problem of finding estimates for $Z :\mathbb{R}^3 \to \mathbb{R}^3$ which satisfies
\begin{align}
\operatorname*{curl} Z &= F  \text{ in } \mathbb{R}^3, \label{curl}\\
 \operatorname*{div} Z &= 0 \text{ in } \mathbb{R}^3\label{div}.
\end{align}
That such a system has a solution is standard, indeed, one can compute that
\begin{align*}
Z=\operatorname*{curl} (-\Delta)^{-1} F
\end{align*}
satisfies \eqref{curl} and \eqref{div}.  When $F \in L^p(\mathbb{R}^3;\mathbb{R}^3)$ and $p>1$,
Calder\'on-Zygmund theory yields
\begin{align}\label{CZ}
\|\nabla Z\|_{L^p(\mathbb{R}^3;\mathbb{R}^{3\times 3})} \leq C\|F\|_{L^p(\mathbb{R}^3;\mathbb{R}^3)},
\end{align}
and therefore for any $1<p<3$ by Sobolev's embedding one has
\begin{align}\label{Sobolev}
\|Z\|_{L^{p^*}(\mathbb{R}^3;\mathbb{R}^{3})} \leq C\|F\|_{L^p(\mathbb{R}^3;\mathbb{R}^3)}.
\end{align}
Yet the inequality \eqref{CZ} is false for $p=1$ while \eqref{Sobolev} persists.

The proof of the validity of \eqref{Sobolev} in the case $p=1$ is due to J. Bourgain and H. Brezis, who in their pioneering papers \cite{BourgainBrezis2004, BourgainBrezis2007} establish such an estimate, along with its analogue for the torus (under the further compatibility condition $\int_{\mathbb{T}^d} F=0$),
\begin{align}\label{estimate}
\| Z\|_{L^{3/2}(\mathbb{T}^3;\mathbb{R}^3)} &\leq C \|F\|_{L^1(\mathbb{T}^3;\mathbb{R}^3)},
\end{align}
among other results.  It was an open question in \cite[Open Problem 1 on p.~295]{BourgainBrezis2007} whether one can improve this estimate on the Lorentz scale, that is, does the solution $Z=\operatorname*{curl} (-\Delta_{\mathbb{T}^3})^{-1} F$ admit the estimate
\begin{align*}
\| Z\|_{L^{3/2,1}(\mathbb{T}^3;\mathbb{R}^3)} &\leq C \|F\|_{L^1(\mathbb{T}^3;\mathbb{R}^3)}?
\end{align*}

In this paper we give an affirmative answer to the Euclidean analogue of this question as a result of a more general phenomena concerning fractional integration.  In particular, for $\alpha \in (0,d)$ we define the Riesz potential of order $\alpha$ acting on $f \in L^1(\mathbb{R}^d)$ by
\begin{align}\label{riesz_potentials}
I_\alpha f(x) := \frac{1}{\Gamma(\frac{\alpha}{2})} \int_0^\infty t^{\frac{\alpha}{2}-1} p_t\ast f(x)\;dt,
\end{align}
where $p_t(x) = 1/(4\pi t)^{d/2} \exp(-|x|^2/4t)$ is the heat kernel.  Note that $I_2=(-\Delta)^{-1}$  is the Newtonian potential when $d\geq 3$.  Then we first establish
\begin{theorem}\label{mainresult}
Let $d\geq 2$ and $\alpha \in (0,d)$.  There exists a constant $C=C(\alpha,d)>0$ such that
\begin{align}\label{potentialnodiracl1'}
\|I_\alpha F \|_{L^{d/(d-\alpha),1}(\mathbb{R}^d;\mathbb{R}^d)} \leq C \|F\|_{L^1(\mathbb{R}^d;\mathbb{R}^d)}
\end{align}
for all fields $F \in L^1(\mathbb{R}^d;\mathbb{R}^d)$ such that $\operatorname*{div} F=0$ in the sense of distributions.
\end{theorem}

We then show how Theorem \ref{mainresult} implies
\begin{theorem}\label{bbq}
Suppose $F \in L^1(\mathbb{R}^3;\mathbb{R}^3)$.  The function $Z=\operatorname*{curl} (-\Delta)^{-1} F$ satisfies
\begin{align*}
\operatorname*{curl} Z &= F   \\
 \operatorname*{div} Z &= 0
\end{align*}
in the sense of distributions and there exists a constant $C>0$ such that
\begin{align*}
\| Z\|_{L^{3/2,1}(\mathbb{R}^3;\mathbb{R}^3)} \leq  C\| F\|_{L^{1}(\mathbb{R}^3;\mathbb{R}^3)}.
\end{align*}
In particular, one has
\begin{align*}
\frac{Z(x)}{|x-y|} \in L^1(\mathbb{R}^3;\mathbb{R}^3)
\end{align*}
for every $y \in \mathbb{R}^3$.
\end{theorem}

\begin{remark}
A basic example of an application of Theorem \ref{bbq} is to Maxwell's equations in the static regime, where $Z=B$ is the magnetic field and $F=J$ is the electric current density.  Then the assumption $J$ has finite total mass is the natural one, as this says the total current is finite, in which case one deduces sharp Lorentz regularity of the magnetic field.
\end{remark}

Let us make a few remarks concerning the theorems and an application of Theorem \ref{mainresult} to another PDE before we return to discuss the proofs.  First, that one has the inequality
\begin{align}\label{alsotrueLorentz}
\|I_\alpha F \|_{L^{d/(d-\alpha),q}(\mathbb{R}^d;\mathbb{R}^d)} \leq C \|F\|_{L^1(\mathbb{R}^d;\mathbb{R}^d)}
\end{align}
for all $F \in L^1(\mathbb{R}^d;\mathbb{R}^d)$ such that $\operatorname*{div} F=0$ for any $q>1$ can be deduced from J. Van Schaftingen's result \cite{VS} (and this implies an analogous inequality in Theorem \ref{bbq}).  Second, the validity of Theorem \ref{mainresult} for $d=2$ already follows from work of the second named author in \cite{Spector1}, as in two dimensions the assumption that $\operatorname*{curl} F=0$ is equivalent to $\operatorname*{div} F=0$.  Third, the improvement of the Euclidean analogue of the differential inequality \eqref{estimate} to the optimal result on Lorentz scale is the first such improvement for differential operators which admit a singularity of dimension $<d-1$; for more historical details on this we refer to Section \ref{historical_remarks} below.
Fourth, concerning the precise question of Bourgain and Brezis, the case of the torus; from the perspective of our method this turns out to be slightly more subtle and will be treated in a forthcoming work.  Finally, concerning a refinement of the estimates for Poisson's equation that have been considered in the beginning of \cite{BourgainBrezis2007}, we prove
\begin{theorem}\label{bbq1}
Let $d \geq 3$ and suppose $F \in L^1(\mathbb{R}^d;\mathbb{R}^d)$ satisfies $\operatorname*{div}F=0$ in the sense of distributions.  The function $U=I_2 F$ satisfies
\begin{align*}
-\Delta U = F
\end{align*}
in the sense of distributions and there exist constants $C,C'>0$ such that
\begin{align*}
\| U\|_{L^{d/(d-2),1}(\mathbb{R}^d;\mathbb{R}^d)} &\leq  C\| F\|_{L^{1}(\mathbb{R}^d;\mathbb{R}^d)} \\
\| \nabla U\|_{L^{d/(d-1),1}(\mathbb{R}^d;\mathbb{R}^d)} &\leq  C'\| F\|_{L^{1}(\mathbb{R}^d;\mathbb{R}^d)}.
\end{align*}
\end{theorem}

We conclude the introduction with a discussion of the idea of the proofs.  From the results of the second named author  \cite{Spector1, KrantzPelosoSpector, Spector2}, one finds that in the curl free case to establish the estimate \eqref{potentialnodiracl1'} there are two basic steps:
\begin{enumerate}
\item  Reduce the general inequality to the case of a fundamental object.
\item  Prove the inequality for the fundamental object.
\end{enumerate}
The fundamental objects in that setting are Radon measures $D\chi_E$ where $\chi_E$ is the characteristic function of a set and $D\chi_E$ denotes its distributional derivative; that estimates for such objects suffices is a consequence of the coarea formula.  The proof of the inequality for these objects follows from a single pointwise estimate:  For $\alpha \in (0,1)$, one has the existence of a constant $C=C(\alpha,d)>0$ such that
\begin{align}\label{interpolation}
|I_\alpha D\chi_E | \leq C \left(\sup_{t>0} |p_t \ast D\chi_E |\right)^{1-\alpha} \left(\sup_{t>0}  t^{1/2}  |\nabla p_t \ast \chi_E| \right)^\alpha
\end{align}
for all $\chi_E \in BV(\mathbb{R}^d)$, where $ p_t$ denotes the heat kernel and $\nabla p_t$ its classical gradient.  In particular, from this one easily obtains
\begin{align}
|I_\alpha D\chi_E| &\leq C' \left(\sup_{t>0} |p_t \ast D \chi_E|\right)^{1-\alpha}  \label{local}\\
|I_\alpha D\chi_E| &\leq C \left(\sup_{t>0} |p_t \ast D \chi_E|+\sup_{t>0} |t^{1/2}\nabla p_t \ast \chi_E|\right). \label{global}
\end{align}
The former estimate follows from \eqref{interpolation} by the boundedness of the maximal function involving the derivative of the heat kernel on $L^\infty(\mathbb{R}^d)$ and the fact that $\chi_E$ is bounded; it is above the natural scaling exponent.  The latter estimate follows from Young's inequality; it is below the scaling exponent.  That these are sufficient to obtain an estimate for $I_\alpha D\chi_E$ in the optimal Lorentz space follows from an argument which interpolates these two non-linear estimates appropriately.

An attempt to apply this program in the divergence free case leads to the questions
\begin{enumerate}
\item  What are the fundamental objects?
\item  What inequalities can be established for any claimed fundamental object?
\end{enumerate}
A natural beginning here is a result of S. Smirnov \cite{Smirnov}, which yields a representation of divergence free function in terms of curves.  In particular, one can show that Smirnov's \cite[Theorem A]{Smirnov} implies the following result (which is a variant of an assertion of Bourgain and Brezis \cite[p.~278]{BourgainBrezis2007}):  Given $F \in L^1(\mathbb{R}^d;\mathbb{R}^d)$ such that $\operatorname*{div}F=0$, there exist oriented $C^1$ closed curves $\left\{\Gamma_{i,l}\right\}_{\{1,\ldots, n_l\} \times \mathbb{N}}$ such that
\begin{align}\label{weak-star-convergence}
\int_{\mathbb{R}^d}  \Phi \cdot dF = \lim_{l \to \infty} \frac{\|F\|_{L^1(\mathbb{R}^d;\mathbb{R}^d)}}{n_l \cdot l} \sum_{i=1}^{n_l}  \int_{\mathbb{R}^d} \Phi \cdot \mu_{ \Gamma_{i,l}}
\end{align}
for any $\Phi \in C_0(\mathbb{R}^d;\mathbb{R}^d)$ and
\begin{align}\label{strict-convergence}
\lim_{l \to \infty} \frac{1}{n_l \cdot l} \sum_{i=1}^{n_l}  ||\mu_{\Gamma_{i,l}}||(\mathbb{R}^d)= 1.
\end{align}
Here and in the sequel we use the notation $\mu_{ \Gamma}$ to denote the vector measure induced by integration along an oriented piecewise $C^1$ curve $\Gamma$, i.e.
\begin{align}\label{measure_definition}
\int_{\mathbb{R}^d} \Phi \cdot \mu_{ \Gamma} := \int_0^{L}\Phi(\gamma(t)) \cdot \dot{\gamma}(t)\;dt
\end{align}
for all $\Phi \in C_0(\mathbb{R}^d;\mathbb{R}^d)$, where $\gamma \in C^1([0,L];\mathbb{R}^d)$ is the parametrization by arclength of the oriented curve $\Gamma$, while $||\mu_{\Gamma}||$ denotes the total variation measure of $\mu_{\Gamma}$.

For such objects, we can argue analogously (see Lemma \ref{pointwise_global} below in Section \ref{lemmas})
 to the curl free case to obtain a counterpart of \eqref{interpolation} for oriented piecewise $C^1$ closed curves, the inequality
\begin{align}\label{global_analogue}
|I_\alpha \mu_\Gamma | \leq C_1 \left(\sup_{t>0} |p_t \ast \mu_\Gamma |\right)^{1-\alpha} \left(\sup_{t>0}  t^{1/2}  |\nabla p_t| \ast ||S|| \right)^\alpha,
\end{align}
where $S \in \mathbb{I}_2(\mathbb{R}^d)$ is an integral current, the generalized minimal surface which satisfies
\begin{align*}
&\partial S = \mu_\Gamma \\
&||S||(\mathbb{R}^d)^{1/2}  \leq c  ||\mu_{\Gamma}||(\mathbb{R}^d)
\end{align*}
and $||S|| \in M_b(\mathbb{R}^d)$ denotes the total variation measure of $S$.

While this estimate yields an adequate replacement for \eqref{global}, the term
\begin{align*}
\sup_{t>0}  t^{1/2}  |\nabla p_t| \ast ||S||
\end{align*}
is never bounded in the case we are considering, $d \geq 3$.  We therefore require an alternative replacement to the inequality \eqref{local}.  One possibility which is in a similar spirit is
\begin{align}\label{local_analogue}
|I_\alpha \mu_\Gamma | \leq C_2 \left(\sup_{t>0} |p_t \ast \mu_\Gamma |\right)^{1-\frac{\alpha}{d-1}}
\|I_{d-1}\mu_\Gamma\|_{BMO} ^\frac{\alpha}{d-1}.
\end{align}
Indeed, if one had the bound
\begin{align}\label{Adamsestimate}
\|I_{d-1}\mu_\Gamma\|_{BMO} \leq C_3
\end{align}
for the measures $\mu_\Gamma$ associated to the curves which arise in \eqref{weak-star-convergence} with some universal constant $C_3>0$, then one would be able to complete the estimate for such curves by an interpolation argument analogous to that in the curl free case.  If $d=2$, the boundedness of the Riesz transforms from $L^\infty$ into $BMO$ implies \eqref{Adamsestimate} as follows:  If $E$ is the set of finite measure which $\Gamma$ encloses then
\begin{align*}
\|I_{1}\mu_\Gamma\|_{BMO} = \|I_1 \nabla^\perp \chi_E\|_{BMO} = \|R \chi_E\|_{BMO}  \leq C\|\chi_E\|_{L^\infty}.
\end{align*}
However, when $d\geq 3$ the inequality \eqref{Adamsestimate} does not necessarily hold for the curves produced by Smirnov's approximation without further qualification.  We therefore need to work a little more.  The idea is that perhaps not every curve from the approximation \eqref{weak-star-convergence} satisfies this, but that it is possible to write any curve as the sum of curves that do.  Here we rely on a result of D.R. Adams \cite{Adams:1975} which shows that the $BMO$ estimate \eqref{Adamsestimate} holds when $\mu_\Gamma$ satisfies the ball growth condition
\begin{align}\label{ballgrowth}
 \sup_{x\in \mathbb{R}^d,r>0} \frac{||\mu_\Gamma||(B_r(x))}{r} \leq C_4
\end{align}
for some $C_4>0$ (and then $C_3$ depends on $C_4$).

Therefore, taking for granted the estimate \eqref{local_analogue}, it suffices for our purposes to show that given an oriented $C^1$ closed curve $\Gamma$ one can decompose it into the sum of oriented piecewise $C^1$ closed curves (in the sense that its associated measure is the sum of associated measures of these curves), each of whose associated measures satisfies \eqref{ballgrowth} with a universal constant, and such that the total length increases by a bounded factor.  We prove such a decomposition in Lemma \ref{surgery-lem} below in Section \ref{surgery}, which along with the proof of \eqref{local_analogue} in Lemma \ref{interpolation1} below in Section \ref{lemmas}, finally yields an answer to our two questions:  The fundamental objects for divergence free measures are oriented piecewise $C^1$ closed curves which satisfy \eqref{ballgrowth} with some universal constant; such fundamental objects admit the estimates \eqref{global_analogue} and \eqref{local_analogue}, which by an interpolation argument analogous to the curl free case yields the desired sharp Lorentz inequality.

That oriented piecewise $C^1$ closed curves which satisfy a ball growth condition are fundamental objects for the space of divergence free measures is itself a result of independent interest that we record as
\begin{theorem}\label{approximation}
Suppose $F \in M_b(\mathbb{R}^d;\mathbb{R}^d)$ satisfies $\operatorname*{div}F=0$ in the sense of distributions.  Then there exist oriented piecewise-$C^1$ closed curves $\{\Gamma_{i,l,j}\}_{\{1,\ldots,n_l\}\times \mathbb{N}\times \{1,\ldots,k_{i}\}}$ for which
\begin{align*}
F= \lim_{l \to \infty} \frac{||F||(\mathbb{R}^d)}{n_l \cdot l} \sum_{i=1}^{n_l}  \sum_{j=1}^{k_{i}}  \mu_{\Gamma_{i,l,j}}\end{align*}
weakly-star as measures,
\begin{align*}
\lim_{l \to \infty} \frac{1}{n_l \cdot l} \sum_{i=1}^{n_l} \sum_{j=1}^{k_i} ||\mu_{\Gamma_{i,l,j}}||(\mathbb{R}^d)
\leq 10,
\end{align*}
and
\begin{align*}
\|\mu_{\Gamma_{i,l,j}}\|_{\mathcal{M}^1(\mathbb{R}^d)}:= \sup_{x\in \mathbb{R}^d,r>0} \frac{||\mu_{\Gamma_{i,l,j}}||(B(x,r))}{r} \leq 1000.
\end{align*}
\end{theorem}
\noindent
Finally we can conclude the discussion of the proofs:  The deduction of Theorem \ref{mainresult} follows easily from Theorem \ref{approximation} and Lemmas \ref{pointwise_global} and \ref{interpolation1}; Theorems \ref{bbq} and \ref{bbq1} in turn follow from Theorem \ref{mainresult} and the boundedness of singular integrals on Lorentz spaces.

The plan of the paper is as follows.  In Section \ref{historical_remarks} we review the literature on estimates for Riesz potentials and Sobolev inequalities in the $L^1$ regime.  In Section \ref{preliminaries} we recall some necessary preliminaries.  In Section \ref{lemmas}, we prove in Lemmas \ref{pointwise_global} and \ref{interpolation1} several potential estimates for vector measures associated to curves, which we will utilize in the proof of Theorem \ref{mainresult}.  In Section \ref{surgery}, we first give a precise statement of Lemma \ref{surgery-lem} - the surgery lemma.  This lemma makes precise the earlier discussion that one can take an oriented $C^1$ closed curve and produce a family of oriented piecewise $C^1$ closed curves decomposing this curve, each of which satisfies a ball growth constant and such that the total length of the curves obtained is at most a bounded factor times the initial length.  We then state and prove several lemmas which provide ball growth conditions for the measures defined in \eqref{measure_definition} under various additional assumptions on the curves.  With these lemmas we then prove the surgery lemma, after which we deduce the atomic decomposition of divergence free measures asserted in Theorem  \ref{approximation}.  Finally, in Section \ref{mainresults} we prove Theorems \ref{mainresult}, \ref{bbq}, and \ref{bbq1}.

\section{Historical Remarks}\label{historical_remarks}
The initiation of the study of the mapping properties of fractional integrals in the spirit\footnote{Hardy and Littlewood's results concern one-sided Riemann-Liouville fractional integrals, with the comment that ``the differences between integrals, whether right or left-handed, with different origins, though not of importance for our purposes here, are not entirely trivial"  \cite[p.~567]{HardyLittlewood}.} of \eqref{riesz_potentials} is due to G.~H.~Hardy and J.~E.~Littlewood \cite{HardyLittlewood}, while for these potentials it was F.~Riesz \cite{Riesz} and S.~Sobolev \cite{sobolev} who obtained the first estimates on the spaces $L^p(\mathbb{R}^d)$: If $\alpha \in (0,d)$ and $1<p<d/\alpha$, then there exists a constant $C>0$ such that
\begin{align*}
\|I_\alpha f \|_{L^{dp/(d-\alpha p)}(\mathbb{R}^d)} \leq C \|f\|_{L^p(\mathbb{R}^d)}
\end{align*}
for all $f \in L^p(\mathbb{R}^d)$.  In the case $p=1$, such an inequality is classically known to fail, cf. \cite[p.~119]{Stein}.  Indeed, if true, by approximation in the strict topology it would necessarily hold for all finite Radon measures, and therefore for a Dirac mass.  However, as this measure is the identity for convolution, this would imply that $I_\alpha \in L^{d/(d-\alpha)}(\mathbb{R}^d)$, and the contradiction becomes evident from the alternative formula for the Riesz kernel
\begin{align*}
I_\alpha(x) \equiv \frac{1}{\gamma(\alpha)} \frac{1}{|x|^{d-\alpha}},
\end{align*}
see, e.g. \cite[Theorem 5 on p.~73]{Stein}.  That the Dirac mass is the only obstruction to such an inequality is an idea which subsequently emerged in the works \cite{BourgainBrezisMironescu,BourgainBrezis2004, BourgainBrezis2007,SSVS,Spector,Spector1,VS,VS2,VS2a,VS3,VS4}, and which we further explore in this paper.

This principle concerning the Dirac mass is manifest, though certainly not explicit, in an early replacement of the $L^1(\mathbb{R}^d)$ inequality, a Hardy space estimate of E. Stein and G. Weiss in \cite{SteinWeiss}:  For any $\alpha \in (0,d)$ there exists a constant $C>0$ such that
\begin{align}\label{SW}
\|I_\alpha f \|_{L^{d/(d-\alpha )}(\mathbb{R}^d)} \leq C \|f\|_{\mathcal{H}^1(\mathbb{R}^d)}
\end{align}
for all $f \in \mathcal{H}^1(\mathbb{R}^d)$, where
\begin{align*}
\|f\|_{\mathcal{H}^1(\mathbb{R}^d)}:= \|f\|_{L^1(\mathbb{R}^d)}+ \|Rf\|_{L^1(\mathbb{R}^d;\mathbb{R}^d)}
\end{align*}
and $Rf:=\nabla I_1f$ is the vector Riesz transform.  From the perspective of ruling out singularities concentrated at a point as the limit of norm bounded sequences, this is a rather strong imposition on the class of functions, as not only does the Hardy space not contain Dirac masses, it does not contain measures concentrated on any lower dimension.  Indeed, one of the results of Stein and Weiss in \cite{SteinWeiss} asserts that if one has a Radon measure whose Riesz transforms are Radon measures, then in fact the Radon measure is absolutely continuous with respect to the Lebesgue measure.  In particular, the result of Stein and Weiss shows that a space which contains no lower dimensional concentrations admits a fractional integration inequality.

Thus one may wonder whether it is possible to find other subspaces of $L^1(\mathbb{R}^d)$ whose closure in the strict topology does not contain measures supported at a point, though which possibly contains measures supported on sets of various dimensions $<d$, which admit such an inequality.  Without further qualification this seems hopeless, though if we view the problem from a slightly different perspective some structure emerges.  In particular, in the spirit of Stein and Weiss's approach we may identify functions in the Hardy space as elements in the space $L^1(\mathbb{R}^d;\mathbb{R}^{d+1})$, and not only that, to view the Hardy space itself as a closed subspace of this vector-valued $L^1$ space.  This, in turn, allows us to view their embedding as a statement about an embedding for a closed subspace of $L^1(\mathbb{R}^d;\mathbb{R}^{d+1})$.  This raises the question of which spaces arising as the closure of subspaces of $L^1(\mathbb{R}^d;\mathbb{R}^{k})$ in the strict topology support an inequality in the spirit of \eqref{SW}.  A full answer to this question is lacking at present (see \cite[Definition 1.1 and Conjecture 1]{ASW} for a conjecture in this direction), though for differentially constrained subspaces one has an elegant characterization due to Jean Van Schaftingen.  Precisely, in \cite[Definition 1.3 on p.~881]{VS3}, Van Schaftingen  introduced the notion of a cocancelling operator:  A homogeneous linear differential operator $L(D):C^\infty_c(\mathbb{R}^d;\mathbb{R}^k) \to C^\infty_c(\mathbb{R}^d;\mathbb{R}^l)$ is cocancelling if
\begin{align*}
\bigcap_{\xi\in\mathbb{R}^d \setminus \{0\}}\ker L(\xi)=\{0\}.
\end{align*}
He then proved in Proposition 8.7 in \cite{VS3} a result which implies that one can have the inequality
\begin{align}\label{false1}
\|I_\alpha F \|_{L^{d/(d-\alpha)}(\mathbb{R}^d;\mathbb{R}^k)} \leq C \|F\|_{L^1(\mathbb{R}^d;\mathbb{R}^k)}
\end{align}
for all $F \in L^1(\mathbb{R}^d;\mathbb{R}^k)$ such that $L(D)F=0$ if and only if $L$ is cocancelling.   The connection of the cocancelling condition with the Dirac mass counterexample is that the former restricts the class of $F$ in the inequality to a subspace of $L^1(\mathbb{R}^d;\mathbb{R}^k)$ for which the strict limit of any norm bounded sequence does not concentrate on any zero dimensional subset, a fact which is evident when one examines the alternative description given by Proposition 2.1 in \cite{VS3}, that one has the equality
 \begin{align*}
		\bigcap_{\xi\in\mathbb{R}^d \setminus \{0\}}\ker L(\xi)=\left\{e\in \mathbb{R}^d\;\colon\; L(\delta_0 e)=0\right\}.
\end{align*}
In short, the imposition of a differential constraint which rules out concentrations at a point, the classical counterexample to such an inequality, is necessary and sufficient for the validity of the inequality.

Two typical examples of cocancelling operators are $L(D) = \operatorname*{curl}$ and $L(D) = \operatorname*{div}$.  For the former one has that such vector functions are curl-free and therefore gradients.  Thus, as a consequence of the coarea formula one deduces that they admit concentrations on sets no smaller than dimension $(d-1)$.  For the latter, a result of M. Roginskaya and M. Wojciechowski \cite{RW} shows that limits of bounded sequences of this subspace may not concentration on sets of dimension smaller than $1$.  In either case, one has $k=d$ and the inequality \eqref{false1} reads
\begin{align}\label{false}
\|I_\alpha F \|_{L^{d/(d-\alpha)}(\mathbb{R}^d;\mathbb{R}^d)} \leq C \|F\|_{L^1(\mathbb{R}^d;\mathbb{R}^d)}
\end{align}
for all $F \in L^1(\mathbb{R}^d;\mathbb{R}^d)$ such that $\operatorname*{curl}F=0$ or $\operatorname*{div}F=0$.  The case $\operatorname*{curl}F=0$ is actually connected to the inequality of Stein and Weiss, as was observed in the paper of the second named author \cite[Inequalities (1.1) and (1.3) on p.~2]{Spector1}, where it was shown that the inequality \eqref{false} is equivalent to the estimate
\begin{align}\label{SSVS_inequality}
\|I_\alpha f \|_{L^{d/(d-\alpha )}(\mathbb{R}^d)} \leq C \|Rf\|_{L^1(\mathbb{R}^d;\mathbb{R}^d)}
\end{align}
for all $f \in C^\infty_c(\mathbb{R}^d)$ (see also \cite{SSVS}, where the inequality \eqref{SSVS_inequality} was first observed).  In particular, \eqref{SSVS_inequality} is a version of \eqref{SW} with the $L^1(\mathbb{R}^d)$ norm removed, hence a stronger inequality.

The discovery of the admissibility of \eqref{false1} under further conditions has its origins in a paper of J. Bourgain, H. Brezis, and P. Mironescu \cite{BourgainBrezisMironescu}, where using Littlewood-Paley theory they proved the following inequality:  There exists a constant $C_d>0$ such that
\begin{align}
\left| \int_{\mathbb{R}^d} Y \cdot d\mu_{\Gamma} \right| \leq C_d ||\mu_{\Gamma}||(\mathbb{R}^d) \|\nabla Y\|_{L^{d}(\mathbb{R}^d;\mathbb{R}^d)} \label{BBM}
\end{align}
for every closed rectifiable curve $\Gamma \subset \mathbb{R}^d$ and every $Y \in \dot{W}^{1,d}(\mathbb{R}^d)$.  By the definition of the Riesz potential and using the boundedness of the Riesz transforms on $L^d(\mathbb{R}^d)$, this is equivalent to the integral estimate
\begin{align}
\left| \int_{\mathbb{R}^d} I_1Y \cdot d\mu_{\Gamma} \right| \leq C'_d ||\mu_{\Gamma}||(\mathbb{R}^d) \| Y\|_{L^{d}(\mathbb{R}^d;\mathbb{R}^d)} \label{BBMprime}
\end{align}
for every $Y \in L^d(\mathbb{R}^d;\mathbb{R}^d)$.  When one takes into an account a result of S. Smirnov \cite{Smirnov}, which can be shown to imply that linear combinations of such circulation integrals are dense in the space of solenoidal fields in the strict topology, one finds a further equivalent formulation in the inequality
\begin{align}\label{true}
\|I_1 F \|_{L^{d/(d-1)}(\mathbb{R}^d;\mathbb{R}^d)} \leq C \|F\|_{L^1(\mathbb{R}^d;\mathbb{R}^d)}
\end{align}
for all $F \in L^1(\mathbb{R}^d;\mathbb{R}^d)$ such that $\operatorname*{div} F=0$ in the sense of distributions.  This is precisely \eqref{false} under the constraint that $\operatorname*{div} F=0$ in the case $\alpha=1$, while the full generality of \eqref{false1} for $\alpha \in (0,d)$ can be argued\footnote{Proposition 8.7 implies the case $\alpha \in (0,1)$ which in combination with Sobolev's inequality yields the full range.} by a result of J. Van Schaftingen \cite{VS3}.

A central question in the theory of these embeddings that has not been resolved is that of optimal estimates on the Lorentz scale.  This question is motivated by what is known in the classical theory and has been raised in various forms for differential analogues of the preceding inequalities as an open problem in a number of papers \citelist{ \cite{BourgainBrezis2007}*{Open problem 1} \cite{VS3}*{Open problem 8.3}\cite{VS4}*{Open problem 2}
\cite{VS2a}*{Open problem 2}\cite{Spector-VanSchaftingen-2018}*{Questions 1.1 and 1.2}}.  Here we recall that around the same time as Stein and Weiss's result concerning an integral replacement of the failure of Sobolev's inequality for $L^1(\mathbb{R}^d)$, E. Gagliardo \cite{Gagliardo} and L. Nirenberg \cite{Nirenberg} proved an alternative replacement in the form of the differential inequality
\begin{align}\label{GN}
\|u\|_{L^{d/(d-1)}(\mathbb{R}^d)} \leq C\|\nabla u \|_{L^1(\mathbb{R}^d;\mathbb{R}^d)}
\end{align}
for all $u \in W^{1,1}(\mathbb{R}^d)$.   This was then refined on the Lorentz scale by A. Alvino in \cite{alvino}, where it was proved that one has the inequality
\begin{align}\label{A_ineq}
\|u\|_{L^{d/(d-1),1}(\mathbb{R}^d)} \leq C\|\nabla u \|_{L^1(\mathbb{R}^d;\mathbb{R}^d)}
\end{align}
for all $u \in W^{1,1}(\mathbb{R}^d)$.  Here the inclusion
\begin{align*}
L^{d/(d-1),1}(\mathbb{R}^d) \subset L^{d/(d-1),d/(d-1)}(\mathbb{R}^d) = L^{d/(d-1)}(\mathbb{R}^d) \subset L^{d/(d-1),\infty}(\mathbb{R}^d)
\end{align*}
implies that \eqref{GN} is sharper than \eqref{A_ineq}, while simple examples show that Alvino's result is optimal.  However, the gradient is not the only operator for which one has such a Lorentz space improvement.   Indeed, one also has such an improvement for the symmetric part of the gradient in the vector-valued setting, along with a class of operators in this spirit.  We here refer to firstly the classical work of  M.~J.~Strauss \cite{Strauss}, where it was proved that one has the Korn-Sobolev inequality
\begin{align*}
\|u\|_{L^{d/(d-1)}(\mathbb{R}^d;\mathbb{R}^d)} \leq C\|\nabla u + (\nabla u)^t \|_{L^1(\mathbb{R}^d;\mathbb{R}^{d\times d})}
\end{align*}
for all $u \in L^1(\mathbb{R}^d;\mathbb{R}^d)$ such that $\nabla u + (\nabla u)^t \in L^1(\mathbb{R}^d;\mathbb{R}^{d\times d})$, and secondly to its Lorentz space refinement  \cite{Spector-VanSchaftingen-2018}, where it was proved that one has the inequality
\begin{align*}
\|u\|_{L^{d/(d-1),1}(\mathbb{R}^d;\mathbb{R}^d)} \leq C\|\nabla u + (\nabla u)^t \|_{L^1(\mathbb{R}^d;\mathbb{R}^{d\times d})}
\end{align*}
for all $u \in L^1(\mathbb{R}^d;\mathbb{R}^d)$ such that $\nabla u + (\nabla u)^t \in L^1(\mathbb{R}^d;\mathbb{R}^{d\times d})$.  As noted, the latter paper also contains analogous inequalities for a more general class of operators, though like the gradient and the symmetric part of the gradient, all of the differential operators which it treats, and therefore all that were known to admit the optimal Lorentz inequality before this paper, are those for which bounded sequences admit concentrations on sets of dimension no smaller than $(d-1)$.

A simple example of an inequality valid on the Lebesgue scale whose differential operator admits concentration on sets strictly smaller than $(d-1)$ is due to J. Bourgain and H. Brezis in \cite{BourgainBrezis2007}, that one has
\begin{align}\label{bb_eq}
\|Z\|_{L^{3/2}(\mathbb{R}^3;\mathbb{R}^3)} \leq C \|\operatorname*{curl}Z\|_{L^{1}(\mathbb{R}^3;\mathbb{R}^3)} \end{align}
for all $Z \in L^{1}(\mathbb{R}^3;\mathbb{R}^3)$ such that $\operatorname*{div}Z=0$. In particular, the fact that $\operatorname*{curl}Z$ is divergence free implies, by the result of Roginskaya and Wojciechowski \cite[Theorem 3 on p.~218]{RW}, that any measure obtained as the weak-star limit of a bounded sequence of such functions may concentrate on sets of dimension $1$, but not smaller.  Moreover, the limiting measure $\operatorname*{curl}Z = \mu_{\Gamma}$ for an oriented smooth closed curve $\Gamma$ shows that this result is optimal.  Such an inequality with one dimensional concentrations is in a sense the strongest possible result among the class of operators\footnote{For the differential inequalities we speak of, the notion is not cocancelling, but the dual notion of cancelling operators introduced by Van Schaftingen in \cite[Definition 1.2 on p.~880]{VS3}.} introduced by Van Schaftingen, since the inequality fails if one allows for zero dimensional singularities, and there is no more room between one and zero.

Concerning results for integral inequalities, whether the optimal Lorentz inequality holds for all cocancelling operators is not known (and an answer to this would resolve all of the open problems mentioned above).  In fact, prior to this paper there were only two optimal Lorentz embeddings known for the Riesz potentials.  The first is that one can improve the result of Stein and Weiss to the optimal estimate on the Lorentz scale, a result which has been argued by J. Dorronsoro \cite[p.~1032]{Dorronsoro}.  The second is the curl free case of the inequality \eqref{false} \cite[Theorem 1.1]{Spector1}, that one has
\begin{align}\label{potentialnodiracl1}
\|I_\alpha F \|_{L^{d/(d-\alpha),1}(\mathbb{R}^d;\mathbb{R}^d)} \leq C \|F\|_{L^1(\mathbb{R}^d;\mathbb{R}^d)}
\end{align}
for all $F \in L^1(\mathbb{R}^d;\mathbb{R}^d)$ such that $\operatorname*{curl}F=0$.  These results are of comparable strength to the known results for differential inequalities in that the Hardy space estimate does not admit any concentrations, while the curl-free case admits concentrations on $(d-1)$-dimensional sets.  In parallel to the question for their differential analogues, a natural question is whether one can obtain such optimal estimates for all cocancelling operators, or even for a single cocancelling operator which admits singularities of $F$ of dimension $<d-1$.  The question of whether $F \in L^1(\mathbb{R}^d;\mathbb{R}^d)$ such that $\operatorname*{div} F=0$ satisfies the optimal Lorentz inequality \eqref{potentialnodiracl1} was raised by Haim Brezis at the conclusion of a lecture of the second named author at Rutgers in January, 2019, and was the impetus for this paper.

\section{Preliminaries}\label{preliminaries}
In this paper we begin with $F \in L^1(\mathbb{R}^d;\mathbb{R}^d)$ such that $\operatorname*{div}F=0$, though the result holds more generally for vector-valued Radon measures $F \in M_b(\mathbb{R}^d;\mathbb{R}^d)$ such that $\operatorname*{div}F=0$, which is also more useful to treat for pedagogical reasons.  Thus, we will here consider vector-valued Radon measures $F$ which are solenoidal (have zero divergence).  It will be useful to view these objects and some closely related objects from several perspectives.  In particular, given such an $F$ one can write it in terms of its components
\begin{align*}
F=(F_1,F_2,\ldots,F_d),
\end{align*}
where each $F_j \in M_b(\mathbb{R}^d;\mathbb{R})$.  For such finite ``charges" (in the terminology of \cite{Smirnov}), one can define the total variation measure of $F$, $||F||\in M_b(\mathbb{R}^d)$, the scalar Radon measure defined for Borel sets $E\subset \mathbb{R}^d$ by
\begin{align*}
||F||(E):= \sup_{E= \cup_k E_k} \sum_k |F(E_k)|
\end{align*}
where $|F(E_k)|$ denotes the Euclidean norm of the vector obtained by integration of each of the components $F_j$ over $E_k$ and $\{E_k\}_{k \in \mathbb{N}}$ is a partition of $E$ into Borel subsets.

The assumption that $F$ is solenoidal implies that $F$ is a normal current, i.e. $F \in \mathbb{N}_1(\mathbb{R}^d)$.  Among other results, one of the achievements of Smirnov in \cite{Smirnov} is to prove that such objects admit an integral representation over curves of length $l$.  His results can be shown to imply that there exist oriented $C^1$ closed curves $\Gamma_{i,l}$ with induced measures $\mu_{\Gamma_{i,j}}$ as defined by \eqref{measure_definition} such that
\begin{align*}
\int_{\mathbb{R}^d}  \Phi \cdot dF = \lim_{l \to \infty} \frac{||F||(\mathbb{R}^d)}{n_l \cdot l} \sum_{i=1}^{n_l}  \int_{\mathbb{R}^d} \Phi \cdot \mu_{\Gamma_{i,l}}
\end{align*}
with
\begin{align*}
\lim_{l \to \infty} \frac{1}{n_l \cdot l} \sum_{i=1}^{n_l}  ||\mu_{\Gamma_{i,l}}||(\mathbb{R}^d)= 1.
\end{align*}

Each of these curves $\Gamma_{i,l}$, as well as those obtained by Lemma \ref{surgery-lem} in Section \ref{surgery}, is an oriented piecewise $C^1$ closed curve and therefore can be identified with an integral current, i.e. (see  \cite[p.~381]{Federer})
\begin{align*}
T:=\mu_\Gamma \in \mathbb{I}_1(\mathbb{R}^d).
\end{align*}
Moreover, as for any $1$-form
\begin{align*}
\partial T(\psi):= T(d\psi) \equiv 0,
\end{align*}
one has by \cite[4.2.10]{Federer} that there exists a generalized minimal surface $S \in \mathbb{I}_2(\mathbb{R}^d)$ such that
\begin{align*}
&\partial S=T \\
&||S||(\mathbb{R}^d)^{1/2} \leq c ||T||(\mathbb{R}^d)=c ||\mu_\Gamma||(\mathbb{R}^d),
\end{align*}
where $||S||$, $||T||$, and $||\mu_\Gamma||$ denote the total variation measures, as above, and $||S||(\mathbb{R}^d)$, $||T||(\mathbb{R}^d)$,  and $||\mu_\Gamma||(\mathbb{R}^d)$ their masses.

The current $S$ can be expressed in coordinates by
\[
S = \sum_{ij} S_{ij} \frac{\partial}{\partial x_i} \otimes
\frac{\partial}{\partial x_j},
\]
where $S_{ij}$ are Radon measures.  In these coordinates, the expression
$p_t\ast \partial S$ can also be written
\[
    p_t\ast \partial S =
    \sum_{ij} p_t \ast \partial_i (S_{ij} - S_{ji})
    \frac{\partial}{\partial x_j}
    =
    \sum_{ij} \frac{\partial p_t}{\partial x_i} \ast (S_{ij} - S_{ji})
    \frac{\partial}{\partial x_j}.
\]
In particular, we have the pointwise bound
\[
    |p_t\ast \partial S|
    \leq c(d) |\nabla p_t| \ast ||S||
\]
for some $c(d)>0$ that depends only on the dimension.

Let us now recall some results concerning the Lorentz spaces $L^{q,r}(\mathbb{R}^d)$, where we follow the development of R. O'Neil in \cite{oneil}.    We begin with some definitions related to the non-increasing rearrangement of a function.
\begin{definition}
For $f$ a measurable function on $\mathbb{R}^d$, we define
\begin{align*}
m(f,y):= |\{ |f|>y\}|.
\end{align*}
As this is a non-increasing function of $y$, it admits a left-continuous inverse, called the non-negative rearrangement of $f$, and which we denote $f^*(x)$.  Further, for $x>0$ we define
\begin{align*}
f^{**}(x):= \frac{1}{x}\int_0^x f^*(t)\;dt.
\end{align*}
\end{definition}
We can now give a definition of the Lorentz spaces $L^{q,r}(\mathbb{R}^d)$.
\begin{definition}
Let $1<q<+\infty$ and $1\leq r<+\infty$.  We define
\begin{align*}
\|f\|_{L^{q,r}(\mathbb{R}^d)} := \left( \int_0^\infty \left[t^{1/q} f^{**}(t)\right]^r\frac{dt}{t}\right)^{1/r},
\end{align*}
and for $1\leq q \leq+\infty$ and $r=+\infty$
\begin{align*}
\|f\|_{L^{q,\infty}(\mathbb{R}^d)} := \sup_{t>0} t^{1/q} f^{**}(t).
\end{align*}
Finally, the Lorentz space $L^{q,r}(\mathbb{R}^d)$ is defined as
\begin{align*}
L^{q,r}(\mathbb{R}^d) := \left\{ f \text{ measurable} : \|f\|_{L^{q,r}(\mathbb{R}^d)} <+\infty\right\}.
\end{align*}

\end{definition}

For such parameters $q,r$, these functionals can be shown to be norms and the associated spaces $L^{q,r}(\mathbb{R}^d)$ Banach spaces  (see, e.g., \cite[Section 2]{Hunt}).  Moreover, one can prove that the dual of $L^{q,r}(\mathbb{R}^d)$ is $L^{q',r'}(\mathbb{R}^d)$ for $1<q<+\infty$ and $1\leq r < +\infty$, so that by the Hahn-Banach theorem one has
\begin{align*}
\| f\|_{L^{q,r}(\mathbb{R}^d)} = \sup \left\{ \left| \int_{\mathbb{R}^d} fg \;dx \right| : g \in L^{q',r'}(\mathbb{R}^d) \;\; \|g \|_{L^{q',r'}(\mathbb{R}^d)}\leq 1\right\}.
\end{align*}

One has a version of H\"older's inequality (this is a variant of Theorem 3.4 in \cite{oneil}, whose statement and proof can be found in \cite[Section 2]{MS}):
\begin{theorem}\label{holder}
Let $f \in L^{q_1,r_1}(\mathbb{R}^d)$ and $g \in L^{q_2,r_2}(\mathbb{R}^d)$, where
\begin{align*}
\frac{1}{q_1}+\frac{1}{q_2}&=\frac{1}{q}<1\\
\frac{1}{r_1}+\frac{1}{r_2}&\geq  \frac{1}{r},
\end{align*}
for some $r \geq 1$.   Then
\begin{align*}
\|fg\|_{L^{q,r}(\mathbb{R}^d)} \leq e^{1/e}q'\|f \|_{L^{q_1,r_1}(\mathbb{R}^d)}\|g \|_{L^{q_2,r_2}(\mathbb{R}^d)}
\end{align*}
\end{theorem}
Concerning estimates for functions in these spaces, a simpler quantity for our purposes is a quasi-norm which does not involve rearrangements:
\begin{align*}
|||f|||_{L^{q,r}(\mathbb{R}^d)} \equiv q^{1/r} \left(\int_0^\infty \left(t |\{ |f|>t\}|^{1/q}\right)^{r} \frac{dt}{t}\right)^{1/r}.
\end{align*}
In particular, one can show this is equivalent to the norm on $ \|f\|_{L^{q,r}(\mathbb{R}^d)}$ (see, e.g. \cite[Section 2]{MS}):
\begin{proposition}
Let $1<q<+\infty$ and $1\leq r \leq +\infty$.  Then
\begin{align*}
|||f|||_{L^{q,r}(\mathbb{R}^d)} \leq \|f\|_{L^{q,r}(\mathbb{R}^d)}\leq q' |||f|||_{L^{q,r}(\mathbb{R}^d)}.
\end{align*}
\end{proposition}

\section{Potential Estimates for Curves}\label{lemmas}

The goal of this section is to substantiate the inequalities \eqref{global_analogue} and \eqref{local_analogue}.

\begin{lemma}\label{pointwise_global}
Let $\alpha \in (0,1)$.  There exists a constant $C_1=C_1(\alpha,d)>0$ such that
\begin{align*}
|I_\alpha \mu_\Gamma | \leq C_1 \left(\sup_{t>0} |p_t \ast \mu_\Gamma |\right)^{1-\alpha} \left(\sup_{t>0}  t^{1/2}  |\nabla p_t| \ast ||S|| \right)^\alpha
\end{align*}
for every oriented piecewise $C^1$ closed curve $\Gamma \subset \mathbb{R}^d$, where $\mu_\Gamma$ is the measure induced by integration along $\Gamma$ defined by \eqref{measure_definition} and $S \in \mathbb{I}_2(\mathbb{R}^d)$ is an integral current, the generalized minimal surface which satisfies
\begin{align*}
&\partial S = \mu_\Gamma \\
&||S||(\mathbb{R}^d)^{1/2}  \leq c  ||\mu_{\Gamma}||(\mathbb{R}^d)
\end{align*}
and $||S|| \in M_b(\mathbb{R}^d)$ denotes the total variation measure of $S$.
\end{lemma}
\begin{proof}[Proof of Lemma \ref{pointwise_global}]
As discussed in Section \ref{preliminaries} we identify
\begin{align*}
T&=\mu_\Gamma \in \mathbb{I}_1(\mathbb{R}^d).
\end{align*}

In particular we have
\begin{align*}
I_\alpha \mu_\Gamma &= \frac{1}{\Gamma(\alpha/2)} \int_0^\infty t^{\alpha/2-1} p_t \ast  T dt \\
&= \frac{1}{\Gamma(\alpha/2)} \int_0^r t^{\alpha/2-1} p_t \ast T \;dt+\frac{1}{\Gamma(\alpha/2)} \int_r^\infty t^{\alpha/2-1} p_t \ast T \;dt\\
&=: I(r)+II(r).
\end{align*}
For $I(r)$, we have
\begin{align*}
|I(r)| &= \left|\frac{1}{\Gamma(\alpha/2)} \int_0^r t^{\alpha/2-1} p_t \ast T \;dt\right|\\
&\leq \frac{1}{\Gamma(\alpha/2)} \sup_{t>0} \left| p_t \ast T \right|  \int_0^r t^{\alpha/2-1}  \;dt \\
&= \frac{1}{\Gamma(\alpha/2+1)} r^{\alpha/2} \sup_{t>0}\left| p_t \ast T\right| .
\end{align*}
Meanwhile, for $II(r)$, we use the fact that $T=\partial S$ to write circulation integral as an integral over the minimal surface spanning $\Gamma$:
\begin{align*}
|II(r)| &= \left|\frac{1}{\Gamma(\alpha/2)} \int_r^\infty t^{\alpha/2-1} p_t \ast  \partial S\; dt \right|.
\end{align*}
Then as the computation in Section \ref{preliminaries} shows
\begin{align*}
|p_t\ast \partial S| \leq c(d) |\nabla p_t| \ast ||S||
\end{align*}
we have
\begin{align*}
|II(r)|
&\leq \frac{c(d)}{\Gamma(\alpha/2)} \frac{r^{\alpha/2-1/2}}{1/2-\alpha/2} \sup_{t>0} t^{1/2} |\nabla p_t| \ast ||S||
\end{align*}
One can then optimize in $r$, though the choice such that the upper bounds we obtain for $I(r)$ and $II(r)$ are equal is sufficient for our purposes, from which we obtain
\begin{align*}
|I_\alpha \mu_\Gamma | \leq C \left(\sup_{t>0}\left|p_t \ast \mu_\Gamma \right|\right)^{1-\alpha} \left(\sup_{t>0} t^{1/2}|\nabla p_t| \ast ||S|| \right)^\alpha
 \end{align*}
with
\begin{align*}
C_1= 2 \frac{1}{\Gamma(\alpha/2+1)^{1-\alpha}} \left(\frac{1}{\Gamma(\alpha/2)} \frac{c(d)}{1/2-\alpha/2}\right)^{\alpha}.
\end{align*}

\end{proof}

\begin{lemma}\label{interpolation1}
Let $\alpha \in (0,d-1)$.  There exists a constant $C_2=C_2(\alpha,d)>0$ such that
\begin{align*}
|I_\alpha \mu_\Gamma | \leq C_2 \left(\sup_{t>0} |p_t \ast \mu_\Gamma |\right)^{1-\frac{\alpha}{d-1}}
\|I_{d-1}\mu_\Gamma\|_{BMO} ^\frac{\alpha}{d-1}
\end{align*}
for every oriented piecewise $C^1$ curve $\Gamma \subset \mathbb{R}^d$, where $\mu_\Gamma$ is the measure induced by oriented integration along $\Gamma$ defined by \eqref{measure_definition}.
\end{lemma}

\begin{proof}[Proof of Lemma \ref{interpolation1}]
We have
\begin{align*}
I_\alpha \mu_\Gamma &= \frac{1}{\Gamma(\alpha/2)} \int_0^\infty t^{\alpha/2-1} p_t \ast  \mu_\Gamma \;dt \\
&= \frac{1}{\Gamma(\alpha/2)} \int_0^r t^{\alpha/2-1} p_t \ast \mu_\Gamma \;dt+\frac{1}{\Gamma(\alpha/2)} \int_r^\infty t^{\alpha/2-1} p_t \ast \mu_\Gamma \;dt\\
&=: I(r)+II(r).
\end{align*}
For $I(r)$, we proceed as in the proof of Lemma \ref{pointwise_global} to obtain
\begin{align*}
|I(r)| &= \left|\frac{1}{\Gamma(\alpha/2)} \int_0^r t^{\alpha/2-1} p_t \ast \mu_\Gamma \;dt\right|\\
&\leq \frac{1}{\Gamma(\alpha/2)} \sup_{t>0} \left|p_t \ast \mu_\Gamma \right| \int_0^r t^{\alpha/2-1}  \;dt \\
&= \frac{1}{\Gamma(\alpha/2+1)} r^{\alpha/2} \sup_{t>0} \left|p_t \ast \mu_\Gamma \right|.
\end{align*}
Meanwhile, for $II(r)$, we introduce the fractional Laplacian and its inverse the Riesz potential on $p_t$, which is justified for $t>0$ as $p_t$ is smooth and decays exponentially:
\begin{align*}
p_t = I_{d-1} (-\Delta)^{(d-1)/2} p_t.
\end{align*}
In particular, Fubini's theorem and a change of variables implies that
\begin{align*}
|II(r)| = \left|\frac{1}{\Gamma(\alpha/2)} \int_r^\infty t^{\alpha/2-1} (-\Delta)^{(d-1)/2} p_t \ast I_{d-1}\mu_\Gamma  \;dt \right|,
\end{align*}
so that by $\mathcal{H}^1-BMO$ duality we find
\begin{align*}
|II(r)| \leq \frac{1}{\Gamma(\alpha/2)} \int_r^\infty t^{\alpha/2-1} \|(-\Delta)^{(d-1)/2} p_t\|_{\mathcal{H}^1(\mathbb{R}^d)} \|I_{d-1}\mu_\Gamma\|_{BMO}\;dt.
\end{align*}
The observation that
\begin{align*}
\|(-\Delta)^{(d-1)/2} p_t\|_{\mathcal{H}^1(\mathbb{R}^d)} = \frac{c}{t^{(d-1)/2}}
\end{align*}
thus yields
\begin{align*}
|II(r)| &\leq\frac{c}{\Gamma(\alpha/2)}\|I_{d-1}\mu_\Gamma\|_{BMO} \int_r^\infty t^{\alpha/2-1-(d-1)/2}\;dt \\
&=\frac{c}{\Gamma(\alpha/2)}\|I_{d-1}\mu_\Gamma\|_{BMO} \frac{r^{\alpha/2-(d-1)/2}}{(d-1)/2-\alpha/2}.
\end{align*}

One can then optimize in $r$, though the choice such that the upper bounds we obtain for $I(r)$ and $II(r)$ are equal is sufficient for our purposes, from which we obtain
\begin{align*}
|I_\alpha \mu_\Gamma | \leq C \left(\sup_{t>0} \left|p_t \ast \mu_\Gamma \right|\right)^{1-\frac{\alpha}{d-1}} \|I_{d-1}\mu_\Gamma\|_{BMO}^\frac{\alpha}{d-1}
 \end{align*}
with
\begin{align*}
C_2= 2 \frac{1}{\Gamma(\alpha/2+1)^{1-\frac{\alpha}{d-1}}} \left(\frac{c}{\Gamma(\alpha/2)} \frac{1}{(d-1)/2-\alpha/2}\right)^\frac{\alpha}{d-1}.
\end{align*}
\end{proof}

\section{The Surgery Lemma and Atomic Decomposition} \label{surgery}
In this section we first state and prove the surgery lemma.  We then prove the atomic decomposition of the space of divergence free measures asserted in Theorem \ref{approximation}.  To this end, let us introduce some notation.  For an oriented piecewise $C^{1}$ (not necessarily closed) curve $\Gamma \subset \mathbb{R}^d$ we denote by $\gamma$ its parameterization by arclength, and in the sequel we identify $\Gamma$ and $\gamma$.  In particular, if $\Gamma$ has length $L$, this means
\begin{enumerate}
\item  $\gamma \in C^{1}([0,L]\setminus \{s_i\}_{i=1}^k;\mathbb{R}^d)$  for some finite set $\{s_i\}_{i=1}^k$,
\item $|\dot{\gamma}(s)|=1$ for every $s\in[0,L]\setminus \{s_i\}_{i=1}^k$.
 \end{enumerate}
 We say that $\Gamma \subset \mathbb{R}^d$ is an oriented piecewise $C^{1}$ {\bf closed} curve if in addition
 \begin{enumerate}
\item[3)]  $\gamma(0)=\gamma(L)$,
 \end{enumerate}
and in this case it will be convenient to use a distance on the interval $[0,L]$ which respects the closeness of points $s \in [0,\epsilon)$ and $t \in (L-\epsilon,L]$.  We introduce the symbol
 \[
    d_\Gamma(s,t) := \min_{k\in\mathbb{Z}} |s - t + kL|
\]
to denote this distance.  Next, recall that for any oriented piecewise $C^1$ curve $\Gamma$, definition \eqref{measure_definition} gives an associated measure $\mu_\Gamma$, which is the measure induced by integration along $\Gamma$.  For the convenience of the reader we recall here equation \eqref{measure_definition}:
\[
    \int \Phi(x) \cdot d \mu_\Gamma =
    \int_0^L \Phi(\gamma(s)) \cdot\dot{\gamma}(t) \;d t.
\]
Finally, we introduce a cut operation $C(\Gamma,t,t')$, which takes as inputs a piecewise $C^1$ closed curve $\Gamma$ of some length $L$ and $t,t'\in [0,L]$, and returns two piecewise $C^1$ closed curves $\Gamma'$ and $G$ such that
\[
    \mu_\Gamma = \mu_{\Gamma'} + \mu_G.
\]
In particular, if we let $x=\gamma(t)$, $y=\gamma(t')$,  denote the images of the points $t,t'$ by $\gamma$, $\Gamma'$ and $G$ are the piecewise $C^1$ closed curves whose arclength parametrizations $\gamma'\in C^{0,1}([0,L+|y-x|-(t'-t)],\mathbb{R}^d)$ and $g\in C^{0,1}([0,|y-x|+(t'-t)],\mathbb{R}^d)$ are given by the formulas
\[
    \gamma'(s) =
    \begin{cases}
        \gamma(s), &s\leq t \\
        x + (s-t) (y-x) |y-x|^{-1}, &t < s < t+|y-x| \\
        \gamma(s - |y-x| + (t'-t)), &s > t+|y-x|
    \end{cases}
\]
and
\[
    g(s) =
    \begin{cases}
        \gamma(s + t), &s < t'-t \\
        y + (s-(t'-t)) (x-y) |x-y|^{-1},
        & t'-t<s < t'-t +|x-y|.
    \end{cases}
\]
We call the two straight line segments introduced in the cut ``bridges" and the endpoints ``corners".

We are now prepared to state the surgery lemma.
\begin{lemma}
    \label{surgery-lem}
Suppose $\Gamma$ is an oriented $C^1$ closed  curve.  There exist oriented piecewise $C^1$ closed curves $\{\Gamma_j\}_{j=1}^{N(\Gamma)}$ with induced measures $\{\mu_{\Gamma_j}\}_{j=1}^{N(\Gamma)}$ defined by \eqref{measure_definition} such that
    \begin{enumerate}
\item \begin{align*}
\mu_\Gamma= \sum_{j=1}^{N(\Gamma)} \mu_{\Gamma_j};
\end{align*}
    \item The total length of the curves obtained in the decomposition satisfies
\begin{align*}
\sum_{j=1}^N ||\mu_{\Gamma_j}||(\mathbb{R}^d) \leq 10 ||\mu_{\Gamma} ||(\mathbb{R}^d) ;
\end{align*}
    \item
    Each $\mu_{\Gamma_j}$ satisfies the ball growth condition
  \begin{align*}
 \sup_{x\in \mathbb{R}^d,r>0} \frac{||\mu_{\Gamma_j}||(B_r(x))}{r} \leq 1000.
\end{align*}
    \end{enumerate}
\end{lemma}

Toward the proof of the surgery lemma, our first observation is that a globally invertible curve admits a ball growth condition which is inversely proportional to its parameter of invertibility.  This is made precise in
\begin{lemma}\label{invertible}
Suppose that $\gamma :[0,L] \to \mathbb{R}^d$ is a piecewise $C^1$ parameterization of a curve $\Gamma \subset \mathbb{R}^d$ (not necessarily closed) and that
\begin{align}\label{uniform_bilip}
        |\gamma(s)-\gamma(t)| \geq \epsilon_0 d_\Gamma(s,t)
\end{align}
for some $\epsilon_0 \in (0,1)$ and for all $s,t \in [0,L]$.  Then $\mu_\Gamma$ satisfies the ball growth condition
    \begin{align*}
        ||\mu_\Gamma||(B_r(x)) \leq 4\epsilon_0^{-1}r.
    \end{align*}
    \end{lemma}
\begin{proof}
Let $x\in\mathbb{R}^d$ and $r>0$.  If  $||\mu_\Gamma||(B_r(x)) > 0$, then there exists $s\in [0,L]$ such that $\gamma(s)\in B_r(x)$.  Then $B_r(x)\subset B_{2r} (\gamma(s))$, and so
    \begin{align*}
        ||\mu_\gamma||(B_r(x)) \leq ||\mu_\gamma||(B_{2r}(\gamma(s)))
        = |\{t\in[0,L] | |\gamma(s)-\gamma(t)| \leq 2r\}|.
    \end{align*}
    Because $\gamma$ satisfies~\eqref{uniform_bilip},
    \[
        |\{t\in[0,L] | |\gamma(s)-\gamma(t)| \leq 2r\}|
        \leq |\{ t\in[0,L] | d_L(s,t) < 2\epsilon_0^{-1} r\}|
        \leq 4\epsilon_0^{-1}r,
    \]
    as desired.
\end{proof}

We next observe that every $C^1$-curve admits a ball growth constant which can be taken proportionally to its length and inversely proportional to its $C^1$ norm, which is the content of the following
\begin{lemma}
    \label{C1-small-ball}
    Let $\gamma\in C^1([0,L];\mathbb{R}^d)$ be a curve parametrized by
    arc length (but not necessarily a closed loop)
    and satisfying the uniform continuity condition
    \[
        |\dot{\gamma}(t) - \dot{\gamma}(t')| \leq \frac{1}{3}.
    \]
    for all $t,t'$ with $|t-t'| \leq \delta$.

    Then $\mu_\Gamma$ satisfies the ball growth condition
    \[
        ||\mu_\Gamma||(B_r(x)) \leq 6\lceil \delta^{-1}L\rceil r.
    \]
    for all $x\in\mathbb{R}^d$.
\end{lemma}
\begin{proof}
We begin by dividing $\Gamma$ into $\lceil \delta^{-1}L\rceil$ pieces of size less than $\delta$,
\begin{align*}
\mu_\Gamma = \bigcup_{i=1}^{\lceil \delta^{-1}L\rceil} \mu_{\Gamma_i}.
\end{align*}
For each $\Gamma_i$ we have
\begin{align*}
\gamma_i(t)-\gamma_i(t') &= \int_t^{t'} \dot{\gamma_i}(s)\;ds \\
&=\int_t^{t'} \dot{\gamma_i}(s)-\dot{\gamma_i}(t)\;ds + (t'-t)\dot{\gamma}(t),
\end{align*}
which because $|t-t'|\leq \delta$, the continuity of $\dot{\gamma}$, and $|\dot{\gamma}|\equiv1$ implies
\begin{align*}
\left|\gamma_i(t)-\gamma_i(t')\right| \geq \frac{2}{3}|t-t'|.
\end{align*}
In particular, by Lemma \ref{invertible} one has
\begin{align*}
||\mu_{\Gamma_i}||(B_r(x)) \leq 6r.
\end{align*}
But then the decomposition implies
\begin{align*}
||\mu_\Gamma||(B_r(x)) &\leq  \sum_{i=1}^{\lceil \delta^{-1}L\rceil} ||\mu_{\Gamma_i}||(B_r(x)) \\
&\leq \sum_{i=1}^{\lceil \delta^{-1}L\rceil} 6r \\
&=6\lceil \delta^{-1}L\rceil r,
\end{align*}
as claimed.
\end{proof}

We next introduce a class of curves are piecewise $C^1$ with scale $\delta$ away from corners and which do not have more than a fixed number of corners in an interval whose length is less than a fixed dilation of $\delta$:
\begin{definition}[$(\delta,\epsilon)$- curve]
    \label{delta-def}
    We say that a piecewise-$C^1$ curve $\Gamma$ with discontinuities in the derivative (corners)
    $\{s_j\}_{j=1}^k$ is a $(\delta,\epsilon)$-curve if the following conditions hold:
\begin{itemize}
    \item \textbf{Uniform continuity of }$\dot{\gamma}$:
        For each $j \in \{1,\ldots, k\}$ and every $t,t' \in (s_j,s_{j+1})$ with $|t-t'|\leq \delta$,
        \[
            |\dot{\gamma}(t) - \dot{\gamma}(t')| \leq \frac{1}{3}.
        \]
    \item \textbf{Spacing condition on corners}:
        For each $t,t'$ with $|t-t'|<\epsilon^{-1}\delta$,
        the interval $(t,t')$ contains strictly fewer than $\lceil \epsilon^{-1}\rceil $ of the $\{s_j\}_{j=1}^k$.
\end{itemize}
\end{definition}

The point of such curves is that they behave similarly to curves which are $C^1$ with scale $\delta$ with respect to the ball growth constant.  In particular, if one additionally assumes invertibility for points $s,t$ such that $d_\Gamma(s,t)\geq \delta$ of a $(\delta,\epsilon)$- curve, one obtains a ball growth constant which is comparable to a large scale invertible $C^1$ curve with scale $\delta$ with no corners.  This is made precise in the following
\begin{lemma}
    \label{full-ball-lem}
    Let $\epsilon \in (0,1)$.  If $\gamma$ is a piecewise-$C^1$ curve which is a $(\delta,\epsilon)$-curve
    and which satisfies
    \begin{equation}
        \label{bilip-2}
        |\gamma(s)-\gamma(t)| > \epsilon d_\Gamma(s,t)
\quad \text{ for all } s,t \text{ with }d_\Gamma(s,t)\geq \delta,
    \end{equation}
    then $\Gamma$ satisfies the ball growth condition
    \[
        ||\mu_\Gamma||(B_r(x)) \leq 72\lceil \epsilon^{-1}\rceil r.
    \]
\end{lemma}
\begin{proof}
Let $x\in\mathbb{R}^d$ and $r>0$.  If $r>\delta/2$ and
    $||\mu_\gamma||(B_r(x)) > 0$, then there exists $s\in [0,L]$ such that
    $\gamma(s)\in B_r(x)$.  Then $B_r(x)\subset B_{2r}(\gamma(s))$, and so
    \[
        ||\mu_\gamma||(B_r(x)) \leq ||\mu_\gamma||(B_{2r}(\gamma(s)))
        = |\{t\in[0,L] | |\gamma(s)-\gamma(t)| \leq 2r\}|.
    \]
    Because $\gamma$ satisfies~\eqref{bilip-2},
    \begin{align*}
        |\{t\in[0,L] | |\gamma(s)-\gamma(t)| \leq 2r\}| &= |\{t\in[0,L] | |\gamma(s)-\gamma(t)| \leq 2r, d_\Gamma(s,t) <\delta\}| \\
        &\;\;+  |\{t\in[0,L] | |\gamma(s)-\gamma(t)| \leq 2r, d_\Gamma(s,t) \geq \delta \}|\\
        &\leq 2 \delta + |\{ t\in[0,L] | d_L(s,t) < 2\epsilon^{-1} r\}|\\
        &\leq 4r+ 4\epsilon^{-1}r\\
        &\leq 8\epsilon^{-1} r.
    \end{align*}
If instead $r<\delta/2$, we again find $s \in[0,L]$ such that $\gamma(s) \in B_r(x)$.  We next define
\begin{align*}
I:= \{t \in [0,L] : d_\Gamma(s,t) <  \epsilon^{-1} \delta\}.
\end{align*}
We claim that if $\gamma(t) \in B_{2r}(\gamma(s))$ then $t \in I$.  Indeed, arguing by contradiction, if $t \notin I$ then $d_\Gamma(s,t) \geq  \epsilon^{-1} \delta \geq \delta$ and therefore by the invertibility condition \eqref{bilip-2} we have
\begin{align*}
|\gamma(s)-\gamma(t)| \geq \epsilon d_\Gamma(s,t) \geq  \delta,
\end{align*}
which says
\begin{align*}
\gamma(t) \notin B_{\delta}(\gamma(s))
\end{align*}
and the result follows as $r<\delta/2$ implies
\begin{align*}
B_{2r}(\gamma(s)) \subset B_{\delta}(\gamma(s)).
\end{align*}
This says that the part of the curve in $B_{2r}(\gamma(s))$ is a subset of the image of two intervals of length at most $\epsilon^{-1}\delta$.  We denote these intervals $I^{+},I^-$, to which we now utilize the spacing condition on $(\delta,\epsilon)$-curves  which ensure that for any $t \in I^+, t'\in I^-$ there are strictly fewer than $\lceil \epsilon^{-1}\rceil $ corners in each of the intervals $(t,s)$, $(s,t')$.  Thus we may find two sets of times $\{\tau^{+}_j\}_{j=1}^K$, $\{\tau^{-}_j\}_{j=1}^K$ with
$K\leq \lceil \epsilon^{-1}\rceil +1$ and such that
\begin{align*}
    I^{+} &= \bigcup_{j=1}^K [\tau^{+} _j,\tau^{+} _{j+1}]
    =: \bigcup_{j=1}^K I^{+} _j \\
        I^{-} &= \bigcup_{j=1}^K [\tau^{-} _j,\tau^{-} _{j+1}]
    =: \bigcup_{j=1}^K I^{+} _j
\end{align*}
and $\gamma|_{I^{+} _j}\in C^1(I^{+} _j;\mathbb{R}^d)$, $\gamma|_{I^{-} _j}\in C^1(I^{-} _j;\mathbb{R}^d)$.  Now using
Lemma~\ref{C1-small-ball} we find
\begin{align*}
    ||\mu_\gamma||(B_{2r}(x))
    &\leq 12r \left(\sum_{j=1}^K (1 + \delta^{-1} (\tau^+_{j+1}-\tau^+_j)) + \sum_{j=1}^K (1 + \delta^{-1} (\tau^-_{j+1}-\tau^-_j)) \right) \\
    &\leq 12r \left((K + \delta^{-1} (\epsilon^{-1}\delta))+(K + \delta^{-1} (\epsilon^{-1}\delta))\right)\\
     &\leq 72 \lceil \epsilon^{-1}\rceil r,
\end{align*}
which proves the claim.
\end{proof}

We are now ready to prove Lemma \ref{surgery-lem}.

\begin{proof}[Proof of Lemma~\ref{surgery-lem}]

Let $\gamma\in C^1([0,L];\mathbb{R}^d)$.  Since $\gamma$ is uniformly continuous on $[0,L]$, there exists some $\delta>0$ such that
for all $s,s'\in[0,L]$ with $d_\gamma(s,s')<\delta$,
\[
    |\dot{\gamma}(s)-\dot{\gamma}(s')| \leq \frac{1}{3}.
\]
In particular, $\Gamma$ is a $(\delta,\infty)$-curve.

We now describe how one implements an iteration of sequences of cuts to $\Gamma$ which results in the family of oriented piecewise $C^1$ closed curves $\{\Gamma_i\}_{i=1}^N$ which satisfy the desired properties.  Let $\epsilon \in (0,\frac{1}{10}]$.  If $\Gamma$ satisfies the conditions of Lemma~\ref{full-ball-lem} with this value of $\epsilon$, we are done.  Otherwise since there are no corners on $\Gamma$, $\Gamma$ is a $(\delta,\epsilon)$-curve, and so its failure to satisfy the conditions of Lemma \ref{full-ball-lem} implies that there exist $s,s'\in[0,L]$
with $d_\Gamma(s,s') \geq \delta$ and
\[
    |\gamma(s)-\gamma(s')| \leq \epsilon d_{\Gamma}(s,s').
\]
In this case, set
\[
    \beta_1 = \min \{ d_{\Gamma}(s,s') \,|\,\, s,s' \in [0,L], d_{\Gamma}(s,s') \geq \delta
    \text{ and } |\gamma(s)-\gamma(s')| \leq \epsilon d_{\Gamma}(s,s')\},
\]
and denote by $t,t'$ the values for which the minimum is attained.  One then performs a cut, $C(\Gamma,t,t')$, the result of which are the curves $\Gamma'$ and $G$.  Here there is a slight ambiguity in the labeling of the outputs of the cut depending on the ordering of $t,t'$.  Without loss of generality we specify that the oriented $C^1$ curve $G$ obtained in this way is the one which contains the piece $\left.\gamma\right|_{[t,t']}$ which satisfies, by minimality of the cut, the large scale invertibility
\begin{align*}
|\gamma(s)-\gamma(s')| > \epsilon d_{\Gamma}(s,s')
\end{align*}
for $s,s' \in [t,t']$.  We claim that for this choice of $G$, $\mu_G$ satisfies a ball growth condition.  To see this, note that $G$ consists of two pieces, $\left.\gamma\right|_{[t,t']}$ and a line segment.  The former is a $(\delta,\epsilon)$-curve, and by minimality of $t,t'$, satisfies the conditions of Lemma~\ref{full-ball-lem}.  Therefore it admits a ball growth constant of $72\lceil \epsilon^{-1}\rceil$.  Meanwhile the segment $[t',t]$ has a ball growth constant of $1$.  Therefore $\mu_G$ admits a ball growth constant of $72\lceil \epsilon^{-1}\rceil+1 \leq 73\lceil \epsilon^{-1}\rceil$.

Concerning the second curve returned by the cut, $\Gamma'$, we observe that this curve has length
\begin{align*}
L'\leq L-(1-\epsilon)\beta_1,
\end{align*}
while the total length of all the curves, including the added bridges, admits the bound
\begin{align*}
TL^1\leq L+2\epsilon \beta_1.
\end{align*}
We then label $\Gamma:=\Gamma'$, $L^1:=L'$ and ask whether $\Gamma$ is a $(\delta,\epsilon)$-curve.  If so, we then ask if the large scale invertibility assumption of Lemma \ref{full-ball-lem} is satisfied.  If so, we are finished.  Otherwise $\Gamma$ is a $(\delta,\epsilon)$-curve, but the large scale invertibility assumption of Lemma \ref{full-ball-lem} is not satisfied, so that we again repeat the preceding cut procedure.  We can continue this process, so long as $\Gamma$ is a $(\delta,\epsilon)$-curve (which it will be until one has performed at least $\lceil \epsilon^{-1}\rceil/2$ cuts), the result of which is a sequence of oriented piecewise $C^1$ closed curves $\{\Gamma_i\}_{i=1}^k$ whose associated measures $\mu_{\Gamma_i}$ admit ball growth constants at most $73\lceil \epsilon^{-1}\rceil$.  Moreover, if $\Gamma$ continues to remain a $(\delta,\epsilon)$-curve, we claim this process will terminate in a finite number of steps.  That this is so follows from the fact that the lengths cut out $\{\beta_i\}_{i=1}^k$ are all greater than or equal to $\delta$, while the length of our curve must remain positive.  In particular, an upper bound of the length of the curve $\Gamma$ after $k$ cuts is
\begin{align*}
L^k\leq L-(1-\epsilon)\sum_{i=1}^k\beta_i,
\end{align*}
which combined with the lower bound on the size of each piece cut out yields
 \begin{align*}
T_1\leq \frac{L}{(1-\epsilon)\delta},
\end{align*}
where we let $T_1$ to denote the number of these (Type I) cuts.  The finiteness of this number ensures that in this case one completes the decomposition in a finite number of steps.  Moreover, in this case the associated measures of the curves satisfy the property $1)$, toward $2)$ one has a bound on the total length of all the curves of at most
\begin{align*}
TL^k&\leq L+2\epsilon \sum_{i=1}^k\beta_i \\
&\leq L +2\epsilon\frac{L}{(1-\epsilon)},
\end{align*}
while the ball growth constant in this case is at most $73\lceil \epsilon^{-1}\rceil$.

It may happen, however, that at some point the condition of being a $(\delta,\epsilon)$-curve is no longer satisfied, in which case the search for pieces which satisfy large scale invertibility is not fruitful, since the conditions of Lemma \ref{full-ball-lem} do not hold and so a piece selected in this way would not necessarily satisfy a ball growth condition.  This is the reason why each time before attempting to perform a Type I cut, one checks to see if $\Gamma$ is a $(\delta,\epsilon)$-curve.  If it is not, instead of looking for violations of large scale invertibility, one must first perform a sequence of operations we call Type II cuts until $\Gamma$ can be restored as a $(\delta,\epsilon)$-curve.

What then is a violation of being a $(\delta,\epsilon)$-curve?  For the curves generated by our cuts, the property of being a $(\delta,\epsilon)$-curve is entirely about the number of corners in a small interval.  That this is so is because $\delta$ is chosen so that the original curve is invertible on scale smaller than $\delta$, while the new curve produced by any cut/surgery replacement consists only of pieces of the original curve and straight line segments separated by corners.  Therefore the local invertibility between corners remains unchanged by the cut process.  In particular, if this condition is violated it is because one can find greater than or equal to $ \lceil \epsilon^{-1}\rceil$ corners in an interval less than $\epsilon^{-1}\delta$.  That is, $\Gamma:=\Gamma'$ from some cut is such that there exist $t,t'$ with $|t-t'| < \epsilon^{-1}\delta$ and the interval $(t,t')$ contains  greater than or equal to$\lceil\epsilon^{-1}\rceil (\geq10)$ corners.  Up to a redefinition of $t,t'$, we may assume that the image of $t,t'$ by $\gamma$ are corner points which contain exactly $\lceil\epsilon^{-1}\rceil-2$ corners between them (if initially there are more corners in a small interval, one simply chooses a smaller interval which contains the requisite number).  We then perform a cut $C(\Gamma,t,t')$, which outputs a loop consisting of this piece containing $\lceil\epsilon^{-1}\rceil-2$ corners and the segment closing it labeled as a $\Gamma_i$ and what remains, together with the segment of opposite orientation, as $\Gamma'$.  That the measures $\mu_{\Gamma_i}$ associated to $\Gamma_i$ obtained by a Type II cut admit a ball growth constant we argue as follows.  We treat first the piece with $\lceil\epsilon^{-1}\rceil-2$ corners, from which the result will follow as $\Gamma_i$ consists of this piece and a line segment.  In analogy with the argument of Lemma \ref{full-ball-lem} we find a set of times $\{\tau_j\}_{j=1}^K$  with
$K= \lceil \epsilon^{-1}\rceil$ (points whose images are the two endpoints and the $\lceil\epsilon^{-1}\rceil-2$ corners in between) and so that
\[
    [t,t'] = \bigcup_{j=1}^K [\tau_j,\tau_{j+1}]
    =: \bigcup_{j=1}^K I_j
\]
and $\gamma_i|_{I_j}\in C^1(I_j;\mathbb{R}^d)$.  Then again Lemma~\ref{C1-small-ball} implies
\[
    ||\mu_\Gamma||(B_{r}(x))
    \leq 12r \sum_{j=1}^K (1 + \delta^{-1} (\tau_{j+1}-\tau_j))
    \leq 12r (\lceil \epsilon^{-1}\rceil+ \delta^{-1} (\epsilon^{-1}\delta)) \leq 24\lceil \epsilon^{-1}\rceil r,
\]
so that with the additional line segment this piece satisfies the ball growth condition with constant at most $24\lceil \epsilon^{-1}\rceil +1\leq 25\lceil \epsilon^{-1}\rceil$.

The result of such a Type II cut is therefore a piece which satisfies a ball growth condition and $\Gamma:=\Gamma'$  which has no additional length and strictly less corners that before.  In particular, if $\Gamma$ has been restored to be a $(\delta,\epsilon)$-curve, one can continue to make Type I cuts.  If not, one should simply perform these Type II cuts until one outputs a $(\delta,\epsilon)$-curve, which necessarily happens in a finite number of these Type II cuts, since each such cut strictly decreases the number of corners, the total number of which is finite (and bounded by twice the number of Type I cuts until that point).

Putting all of this together, our algorithm is as follows.
\begin{breakablealgorithm}
\caption{The Surgery Algorithm}
\begin{algorithmic}[1]
\Statex \textbf{Input}: A $(\delta,\epsilon)$-curve $\Gamma$ with
$\epsilon^{-1}\in\mathbb{N}$, $\epsilon^{-1}\geq 10$.
\Statex \textbf{Output}: A collection $\{\Gamma_i\}$ of curves which satisfies $\mu_\Gamma = \sum \mu_{\Gamma_i}$,  $\sum |\Gamma_i|  \leq 10 |\Gamma|$, $||\mu_{\Gamma_i}||$ with ball growth constant at most $73\epsilon^{-1}$.
\State $i:= 1$;
\State $T_1 := 0$, $T_2 := 0$;
\While{$\Gamma$ does not satisfy the hypotheses of Lemma \ref{full-ball-lem}}
\If{$\Gamma$ is not a $(\delta,\epsilon)$-curve}
  \State Perform a Type II cut to produce a curve $\Gamma_i:=G$ which one argues has ball growth constant $24\epsilon^{-1} +1 \leq 25\epsilon^{-1}$ and the
remainder $\Gamma'$;
  \State $T_2 := T_2+1$;
  \State $\Gamma := \Gamma'$;
  \State $i := i+1$;
\Else
\Comment $\Gamma$ is a $(\delta,\epsilon)$-curve, so $\Gamma$ violates the large-scale invertibility ~\eqref{bilip-2}
    \State Perform a Type I cut to produce a curve $\Gamma_i:=G$ which by Lemma~\ref{full-ball-lem} has a ball growth constant $72\epsilon^{-1} +1 \leq 73\epsilon^{-1}$ and
the remainder $\Gamma'$;
    \State $T_1 := T_1+1$;
  \State $\Gamma := \Gamma'$;
  \State $i := i+1$;
\EndIf
\EndWhile
\State $\Gamma_i := \Gamma$;
\State $N := i$;
\end{algorithmic}
\end{breakablealgorithm}
It only remains to show that the program terminates in a finite number of steps, and that when it does terminate one has the desired bounds.

To see that the program terminates, let us denote by $T_1^k, T^k_2$ to be the number of Type I, Type II cuts after $k$ cuts, respectively, we observe that one has the bounds
 \begin{align*}
T^k_1&\leq \frac{L}{(1-\epsilon)\delta},\\
(\epsilon^{-1}-2)&T^k_2\leq 2T^{k-1}_1.
\end{align*}
The first bound follows in the same manner as before, since each Type I cut removes a length which is at least $(1-\epsilon)\delta$.  The second bound follows from the fact that each Type I cut produces two corners while a Type II cut removes $\epsilon^{-1}-2$ corners.  As increasing sequences of numbers bounded above, $T^k_1,T_2^k$ admits limits $T_1,T_2$ with the bounds
 \begin{align*}
T_1&\leq \frac{L}{(1-\epsilon)\delta}\\
T_2&\leq \frac{2\epsilon}{(1-2\epsilon)}\frac{L}{(1-\epsilon)\delta}.
\end{align*}
Therefore the algorithm terminates in a finite number of steps.

Finally we will be finished when we validate the claim that the sequence of curves $\{\Gamma_i\}_{i=1}^N$ generated in this way satisfy the desired properties.  That one has the decomposition of associated measures asserted in $1)$ holds by construction.  In particular, each time one performs surgery the result is two curves, the sum of whose associated measures is equal to the associated measure of the initial curve, and therefore by finite induction one obtains the decomposition of the initial associated measure in terms of the sequence of associated measures.  That one has the length bound claimed in $2)$ we argue as follows.  As above, one observes that the total length generated by Type I cuts can be bounded above by
\begin{align*}
2\epsilon \sum_{i=1}^N\beta_i \leq 2\epsilon\frac{L}{(1-\epsilon)}.
\end{align*}
For Type II cuts, we can bound the additional length added by multiplying the number of such cuts times twice the length of the maximal segments introduced (one factor of $\epsilon^{-1}\delta$ for each line segment introduced by a Type II cut):
\begin{align*}
2\epsilon^{-1}\delta T_2 &\leq 2\epsilon^{-1}\delta  \frac{2\epsilon}{1-2\epsilon}\frac{L}{(1-\epsilon)\delta} \\
&= \frac{4}{1-2\epsilon}\frac{L}{(1-\epsilon)}.
\end{align*}
The choice of $\epsilon =1/10$ leads to the bound for the total lengths
\begin{align*}
\text{Initial + Type I Added Length + Type II Added Length} =L+\frac{2}{9}L+\frac{50}{9}L \leq 10 L.
\end{align*}
Concerning the ball growth constants of the measures associated to this family of curves:  From the pieces produced by Type I cuts we find a constant of $73 \epsilon^{-1}$; from the pieces produced by Type II cuts we find a constant of $25 \epsilon^{-1}$; the terminal piece satisfies the condition of Lemma \ref{full-ball-lem} and therefore has a constant of $72\epsilon^{-1}$.  The choice $\epsilon =1/10$ yields the claimed bound of $1000$.  This completes the proof.
\end{proof}

We next prove Theorem \ref{approximation}, which is an almost immediate consequence of Bourgain and Brezis's observation concerning Smirnov's result \cite{Smirnov} and Lemma \ref{surgery-lem}.
\begin{proof}[Proof of Theorem \ref{approximation}]
We begin with a variant of an assertion of Bourgain and Brezis, that Smirnov's integral decomposition of divergence free measures \cite[Theorem A]{Smirnov} implies \eqref{weak-star-convergence} and \eqref{strict-convergence}, which we recall here for the convenience of the reader: there exist oriented $C^1$ closed curves $\{\Gamma_{i,l}\}_{\{1,\ldots,n_l\}\times \mathbb{N}}$
such that
\begin{align*}
F= \lim_{l \to \infty} \frac{||F||(\mathbb{R}^d)}{n_l \cdot l} \sum_{i=1}^{n_l}   \mu_{\Gamma_{i,l}}
\end{align*}
weakly-star as measures and
\begin{align*}
\lim_{l \to \infty} \frac{1}{n_l \cdot l} \sum_{i=1}^{n_l} || \mu_{\Gamma_{i,l}}||(\mathbb{R}^d) =1.
\end{align*}
An application of Lemma \ref{surgery-lem} to each $\{\Gamma_{i,l}\}$ produces $k_i:=N$ curves $\{\Gamma_{i,l,j}\}_{j=1}^{k_{i}}$ which satisfy the assertions of Lemma \ref{surgery-lem}.  Combined with the preceding convergences we find
\begin{align*}
F= \lim_{l \to \infty} \frac{||F||(\mathbb{R}^d)}{n_l \cdot l} \sum_{i=1}^{n_l}  \sum_{j=1}^{k_{i}}  \mu_{\Gamma_{i,l,j}}\end{align*}
weakly-star as measures,
\begin{align*}
\lim_{l \to \infty} \frac{1}{n_l \cdot l} \sum_{i=1}^{n_l} \sum_{j=1}^{k_i}  ||\mu_{\Gamma_{i,l,j}}||(\mathbb{R}^d) \leq \lim_{l \to \infty} \frac{1}{n_l \cdot l} \sum_{i=1}^{n_l}  10 ||\mu_{\Gamma_{i,l}}||(\mathbb{R}^d) \leq 10,
\end{align*}
and
\begin{align*}
\|\mu_{\Gamma_{i,l,j}}\|_{\mathcal{M}^1(\mathbb{R}^d)}:= \sup_{x\in \mathbb{R}^d,r>0} \frac{||\mu_{\Gamma_{i,l,j}}||(B(x,r))}{r} \leq 1000.
\end{align*}
This completes the proof of the theorem.
\end{proof}

\section{Fractional Integration and Elliptic Systems}\label{mainresults}
We begin with the proof of Theorem \ref{mainresult}.

\begin{proof}[Proof of the Theorem \ref{mainresult}]
First let us observe that it suffices to prove the inequality for $\alpha \in (0,1)$, since once the theorem has been proved for any given value of $\alpha>0$, we may obtain the result for general $\alpha' \in (\alpha,d)$, by the semi-group property of the Riesz potentials and O'Neil's result concerning convolution in Lorentz spaces, i.e.
\begin{align*}
\left \|I_{\alpha'}  F \right\|_{L^{d/(d-\alpha'),1}(\mathbb{R}^d;\mathbb{R}^d)} &= \left \|I_{\alpha'-\alpha}  I_\alpha F \right\|_{L^{d/(d-\alpha'),1}(\mathbb{R}^d;\mathbb{R}^d)} \\
&\leq C\left \|I_\alpha F \right\|_{L^{d/(d-\alpha),1}(\mathbb{R}^d;\mathbb{R}^d)}.
\end{align*}
We therefore proceed to treat the case $\alpha \in (0,1)$.

Next, we claim that Theorem \ref{approximation} implies that it is sufficient to prove the inequality
\begin{align}\label{sufficient}
\left \|I_\alpha  \mu_{\Gamma} \right\|_{L^{d/(d-\alpha),1}(\mathbb{R}^d;\mathbb{R}^d)} \leq C ||\mu_{\Gamma}||(\mathbb{R}^d)
\end{align}
for every oriented piecewise $C^1$ closed curve $\Gamma\subset \mathbb{R}^d$ whose associated measure $\mu_\Gamma$ admits the bound
\begin{align*}
\sup_{x\in \mathbb{R}^d, r>0} \frac{||\mu_{\Gamma}||(B_r(x))}{r} \leq 1000.
\end{align*}

Indeed, supposing we have established \eqref{sufficient} for such curves, for general $F \in M_b(\mathbb{R}^d;\mathbb{R}^d)$ such that $\operatorname*{div}F=0$ we apply Theorem \ref{approximation} to find oriented piecewise $C^1$ closed curves $\Gamma_{i,l,j}$ for which
\begin{align*}
F= \lim_{l \to \infty} \frac{||F||(\mathbb{R}^d)}{n_l \cdot l} \sum_{i=1}^{n_l}  \sum_{j=1}^{k_{i}}  \mu_{\Gamma_{i,l,j}}\end{align*}
weakly-star as measures,
\begin{align*}
\lim_{l \to \infty} \frac{1}{n_l \cdot l} \sum_{i=1}^{n_l} \sum_{j=1}^{k_i} || \mu_{\Gamma_{i,l,j}}||(\mathbb{R}^d)  \leq \lim_{l \to \infty} \frac{1}{n_l \cdot l} \sum_{i=1}^{n_l}  10 || \mu_{\Gamma_{i,l,j}}||(\mathbb{R}^d)  \leq 10,
\end{align*}
and whose associated measures $\mu_{\Gamma_{i,j,l}}$ satisfy the ball growth condition with uniform constant $1000$.

Then for $\Phi \in C_c(\mathbb{R}^d;\mathbb{R}^d)$, $ \|\Phi\|_{L^{d/\alpha,\infty}(\mathbb{R}^d;\mathbb{R}^d)}  \leq 1$, by Fubini's theorem, the preceding convergences, H\"older's inequality in the Lorentz spaces, and the claimed inequality \eqref{sufficient}, we obtain the chain of inequalities
\begin{align*}
\int_{\mathbb{R}^d}  I_\alpha F \cdot \Phi \;dx &= \lim_{l \to \infty} \frac{||F||(\mathbb{R}^d)}{n_l \cdot l} \sum_{i=1}^{n_l}  \sum_{j=1}^{k_i} \int_{\mathbb{R}^d} \Phi \cdot I_\alpha \mu_{\Gamma_{i,l,j}} \\
&\leq \liminf_{l \to \infty} \frac{||F||(\mathbb{R}^d)}{n_l \cdot l} \sum_{i=1}^{n_l} \sum_{j=1}^{k_i}  e^{1/e} \frac{d}{\alpha}\left\|I_\alpha \mu_{\Gamma_{i,l,j}} \right\|_{L^{d/(d-\alpha),1}(\mathbb{R}^d;\mathbb{R}^d)}  \\
&\leq \liminf_{l \to \infty} \frac{||F||(\mathbb{R}^d)}{n_l \cdot l} \sum_{i=1}^{n_l} 10Ce^{1/e} \frac{d}{\alpha} ||\mu_{\Gamma_{i,l,j}}||(\mathbb{R}^d) \\
&\leq 10Ce^{1/e} \frac{d}{\alpha}||F||(\mathbb{R}^d).
\end{align*}
In particular this yields the inequality
\begin{align*}
\int_{\mathbb{R}^d}  I_\alpha F \cdot \Phi \;dx \leq 10Ce^{1/e} \frac{d}{\alpha}||F||(\mathbb{R}^d),
\end{align*}
from which the desired inequality follows by taking the supremum over such $\Phi$ (as the norm of $I_\alpha F$ in $L^{d/(d-\alpha),1}(\mathbb{R}^d;\mathbb{R}^d)$ is realized as the supremum over such pairings by the Hahn-Banach theorem, which utilizes the fact that $L^{d/(d-\alpha),1}(\mathbb{R}^d;\mathbb{R}^d)$ is a Banach space).

Therefore we proceed to prove the inequality \eqref{sufficient}, which by dilation it suffices to prove for $||\mu_{\Gamma}||(\mathbb{R}^d)
=1$.  Let us prove the estimate for the equivalent (see Section \ref{preliminaries}) quasi-norm
\begin{align*}
\left|\left|\left|I_\alpha  \mu_\Gamma \right|\right|\right|_{L^{d/(d-\alpha),1}(\mathbb{R}^d;\mathbb{R}^d)} &= \int_0^\infty |\{ |I_\alpha  \mu_\Gamma | >s\}|^{(d-\alpha)/d} \;ds \\
&=\int_0^1 |\{ |I_\alpha  \mu_\Gamma | >s\}|^{(d-\alpha)/d} \;ds + \int_1^\infty |\{ |I_\alpha  \mu_\Gamma | >s\}|^{(d-\alpha)/d} \;ds \\
&=: I+II.
\end{align*}
For $I$ we use the lower level set estimate from Lemma \ref{pointwise_global} (which requires $\alpha \in (0,1)$), Young's inequality, and the fact that
\begin{align*}
&\left\{   \sup_{t>0} \left|p_t \ast \mu_\Gamma \right|+ \sup_{t>0}  t^{1/2}  |\nabla p_t| \ast ||S|| >\frac{s}{C_1}\right\} \\
 &\;\;\;\;\subset \{ \sup_{t>0} \left|p_t \ast \mu_\Gamma \right| >\frac{s}{2C_1}\} \cup\{ \sup_{t>0}  t^{1/2}  |\nabla p_t| \ast ||S|| >\frac{s}{2C_1}\}
\end{align*}
 to obtain
\begin{align*}
I &\leq  \int_0^1 \left(\left|\left\{   \sup_{t>0} \left|p_t \ast \mu_\Gamma \right| >\frac{s}{2C_1}\right\}\right| + \left|\left\{ \sup_{t>0}  t^{1/2}  |\nabla p_t| \ast ||S|| >\frac{s}{2C_1}\right\}\right|\right)^{(d-\alpha)/d} \;ds.
\end{align*}
Next, the weak-type estimate for the maximal functions
\begin{align*}
\nu &\mapsto  \mathcal{M}_1(\nu):=\sup_{t>0} |p_t \ast \nu |, \\
\nu &\mapsto  \mathcal{M}_2(\nu):= \sup_{t>0}  t^{1/2}  |\nabla p_t| \ast \nu,
\end{align*}
which follows from the weak-type estimate for the Hardy-Littlewood maximal function and the upper bound given by \cite[Theorem 2 on p.~62]{Stein}, implies
\begin{align*}
I &\leq  \int_0^1 \left(\frac{C_1'}{s}||\mu_\Gamma||(\mathbb{R}^d) + \frac{C_1''}{s} ||S||(\mathbb{R}^d)\right)^{(d-\alpha)/d} \;ds.
\end{align*}
Finally, the general isoperimetric inequality $||S||(\mathbb{R}^d)^{1/2} \leq c||\mu_\Gamma||(\mathbb{R}^d) $ from \cite[4.2.10 on p.~408]{Federer} and the fact that $||\mu_\Gamma||(\mathbb{R}^d) =1$ yields
\begin{align*}
I &\leq \int_0^1 \left(\frac{C_1'}{s} + \frac{c^2C_1''}{s} \right)^{(d-\alpha)/d} \;ds\\
&=: C_5<+\infty.
\end{align*}
since $(d-\alpha)/d<1$.

For $II$ we use Lemma \ref{interpolation1} and the fact that $\Gamma$ satisfies the ball growth condition with constant $1000$ to obtain
\begin{align*}
|I_\alpha  \mu_\Gamma | \leq C_2 \sup_{t>0} \left|p_t \ast \mu_\Gamma \right|^{1-\frac{\alpha}{d-1}} C_3^\frac{\alpha}{d-1}
\end{align*}
This upper bound implies
\begin{align*}
II &\leq \int_1^\infty |\{ C_2 \sup_{t>0} \left|p_t \ast \mu_\Gamma  \right|^{1-\frac{\alpha}{d-1}} C_3^\frac{\alpha}{d-1} >s\}|^{(d-\alpha)/d} \;ds,
\end{align*}
so that again by the weak-type estimate for the maximal function $\mathcal{M}_1$ we deduce the upper bound
\begin{align*}
II &\leq \int_1^\infty \left(\frac{C_2'}{s^{(d-1)/(d-1-\alpha)}}||\mu_\Gamma||(\mathbb{R}^d) \right)^{(d-\alpha)/d}\;ds\\
&=\int_1^\infty \left(\frac{C_2'}{s^{(d-1)/(d-1-\alpha)}}\right)^{(d-\alpha)/d}\;ds\\
&=:C_6<+\infty
\end{align*}
since
\begin{align*}
\frac{d-1}{d-1-\alpha} \times \frac{d-\alpha}{d}>1.
\end{align*}
Thus we have proved that for a curve of length $1$ one has the estimate
\begin{align*}
\|I_\alpha  \mu_\Gamma  \|_{L^{d/(d-\alpha),1}(\mathbb{R}^d;\mathbb{R}^d)} \leq C_5+C_6,
\end{align*}
which was the claim.  This completes the proof.
\end{proof}

We next prove Theorem \ref{bbq}.

\begin{proof}[Proof of Theorem \ref{bbq}]
It is a simple calculation to show that
\begin{align*}
Z= \operatorname*{curl} (-\Delta)^{-1}F \equiv I_2F
\end{align*}
satisfies the equations and so it only remains to prove the estimate.  However, if we denote by $R_i$ the $i$th Riesz transform
\begin{align*}
R_if:= \frac{\partial }{\partial x_i} I_1 f
\end{align*}
then we may express
\begin{align*}
Z = \langle R_2(I_1F_3)-R_3(I_1F_2),R_3(I_1F_1)-R_1(I_1F_3),R_1(I_1F_2)-R_2(I_1F_1)\rangle.
\end{align*}
In particular, the boundedness of $R_i:L^p(\mathbb{R}^3) \to L^p(\mathbb{R}^3)$ for $1<p<+\infty$ implies, by interpolation, its boundedness on the Lorentz spaces $L^{p,q}(\mathbb{R}^3)$ for $1<p<+\infty$ and $1\leq q \leq +\infty$, and therefore
\begin{align*}
\|Z\|_{L^{3/2,1}(\mathbb{R}^3;\mathbb{R}^3)} \leq C\|I_1 F\|_{L^{3/2,1}(\mathbb{R}^3;\mathbb{R}^3)}.
\end{align*}
But then this inequality and an application of Theorem \ref{mainresult} completes the proof of the claimed Lorentz scale inequality.  The fact that one has
\begin{align*}
\frac{Z(x)}{|x-y|} \in L^1(\mathbb{R}^3;\mathbb{R}^3)
\end{align*}
for every $y \in \mathbb{R}^3$ then follows by H\"older's inequality on the Lorentz scale, as $Z \in L^{3/2,1}(\mathbb{R}^3;\mathbb{R}^3)$ and $|x-y|^{-1} \in L^{3,\infty}(\mathbb{R}^3)$, see e.g. \cite[Theorem 3.5]{oneil} for a statement and \cite[Section 2]{MS} for a proof.
\end{proof}

We conclude this section with a proof of Theorem \ref{bbq1}.

\begin{proof}[Proof of Theorem \ref{bbq1}]
As the solution of the vector Poisson equation is given by
\begin{align*}
U=I_2 F,
\end{align*}
the inequality
\begin{align*}
\| U\|_{L^{d/(d-2),1}(\mathbb{R}^d;\mathbb{R}^d)} &\leq  C\| F\|_{L^{1}(\mathbb{R}^d;\mathbb{R}^d)}
\end{align*}
follows immediately from Theorem \ref{mainresult}.  Meanwhile, similar to the proof of Theorem \ref{bbq}, we may express the gradient of $U$ as the vector Riesz transform of the vector-field $F$,
\begin{align*}
\nabla U = R I_1F.
\end{align*}
Therefore, by the boundedness of the Riesz transforms on the Lorentz spaces and an application of Theorem \ref{mainresult}, we find
\begin{align*}
\| \nabla U\|_{L^{d/(d-1),1}(\mathbb{R}^d;\mathbb{R}^{d\times d})} & \leq C\| I_1 F\|_{L^{d/(d-1),1}(\mathbb{R}^d;\mathbb{R}^d)} \\
&\leq  C'\| F\|_{L^{1}(\mathbb{R}^d;\mathbb{R}^d)},
\end{align*}
which completes the proof.
\end{proof}

\section*{Acknowledgements}
The first named author is supported by the Fannie and John Hertz Foundation.  The second named author is supported by the Taiwan Ministry of Science and Technology under research grant number 110-2115-M-003-020-MY3 and the Taiwan Ministry of Education under the Yushan Fellow Program.  The impetus for this paper was a question asked by Haim Brezis after a lecture given by the second named author at Rutgers University.   The second named author would like to warmly thank Haim Brezis and Rutgers University for the invitation and hospitality, while both authors would like to thank Haim Brezis for comments on the exposition in this manuscript.  Additionally, the authors would like to thank Ulrich Menne for discussions regarding minimal surfaces and maximal functions, as well as Jean Van Schaftingen and Jesse Goodman for discussions regarding Smirnov's approximation.  Needless to say the remaining shortcomings are our own.


\bibliographystyle{amsplain}


\end{document}